\documentclass[oneside,reqno]{amsart}
\usepackage[english]{babel}
\usepackage{indentfirst}
\usepackage[T1]{fontenc}
\usepackage[utf8]{inputenc}
\usepackage{paralist}
\usepackage{amssymb}
\usepackage{amsthm}
\usepackage{amsmath}
\usepackage{amscd}
\usepackage{graphicx}
\usepackage[colorlinks=true]{hyperref}
\usepackage[margin=1.3in]{geometry}
\usepackage[OT2,OT1]{fontenc}
\newcommand\cyr
{
\renewcommand\rmdefault{wncyr}
\renewcommand\sfdefault{wncyss}
\renewcommand\encodingdefault{OT2}
\normalfont
\selectfont
}
\DeclareTextFontCommand{\textcyr}{\cyr}

\makeatletter
\def\@settitle
{
\begin{center}
\baselineskip14\p@\relax
\LARGE\@title
\end{center}
}
\makeatother
\theoremstyle{plain} 
\newtheorem{teo}{Theorem}[section] 
 
\newtheorem{lem}[teo]{Lemma}
\newtheorem{prop}[teo]{Proposition}
\newtheorem{defn}[teo]{Definition}
\newtheorem{oss}[teo]{Remark}
\numberwithin{equation}{section} 

\newcommand{\R}{\ensuremath{\mathbb{R}}} 
\newcommand{\Rn}{\ensuremath{\mathbb{R}^n}} 
\newcommand{\Rp}{\ensuremath{\mathbb{R}^{n+1}_+}} 
\newcommand{\N}{\ensuremath{\mathbb{N}}} 
\newcommand{\Fl}{\ensuremath{\mathcal{F}_{\varepsilon}}} 
\newcommand{\Lst}{\ensuremath{L^{2_s^*}(\Rn)}} 
\newcommand{\eps}{\ensuremath{\varepsilon}} 
\newcommand{\ukp}{\ensuremath{(U_k)_+}} 
\newcommand{\al}{\ensuremath{ &\;}} 
\newcommand{\lr}[1]{\left( #1 \right)} 
\newcommand{\lrq}[1]{\left[ #1 \right]} 
 \newcommand{\scp}[1]{\langle #1 \rangle} 
 \newcommand{\cx}{\ensuremath{2_s^*}} 
\newcommand{\eqlab}[1]{\begin{equation}  \begin{aligned}#1 \end{aligned}\end{equation}} 
\newcommand{\bgs}[1]{\begin{equation*} \begin{aligned}#1\end{aligned}\end{equation*}} 
 \newcommand{\syslab}[2] []  {\begin{equation}#1  \left\{\begin{aligned}#2\end{aligned}\right.\end{equation}} 
  \newcommand{\sys}[2][]{\begin{equation*}#1  \left\{\begin{aligned}#2\end{aligned}\right.\end{equation*}}
 \newcommand{\Hsa}{\ensuremath{\dot  H_a^s(\R^{n+1}_+)}}  

\date{}
\title[Green function for the fractional Laplacian]{Some observations on the Green function for the ball in the fractional Laplace framework}
\author {Claudia Bucur, Mar\'ia Medina}
\email{c.bucur@unimelb.edu.au, mamedinad@mat.puc.cl }

\thanks{We would like to thank professors Enrico Valdinoci and Jes\'us Garc\'ia Azorero for their very useful discussions and suggestions.}
\title[A fractional elliptic problem with critical growth and
convex nonlinearities]{A fractional elliptic problem in $\R^n$ with critical growth and
convex nonlinearities}

\begin{document}
\maketitle

 \address{(Claudia Bucur) }{\scshape School of Mathematics and Statistics, 813 Swanston Street, Parkville VIC 3010, Australia} 

\address{(Mar\'ia Medina) }{\scshape Facultad de Matem\'aticas, Pontificia Universidad Cat\'olica de Chile, Avenida Vicu\~na Mackenna 4860, Santiago, Chile}

 \begin{abstract}
\noindent In this paper we prove the existence of a positive solution of the nonlinear and nonlocal elliptic equation in $\Rn$ 
\[ (-\Delta)^s u =\varepsilon h u^q+u^{2_s^*-1} \]
in the convex case $1\leq q<2_s^*-1$, where 
		$ 2_s^*={2n}/({n-2s}) $ is the critical fractional Sobolev exponent, $(-\Delta)^s$ is the fractional Laplace operator, $\varepsilon$ is a small parameter and $h$ is a given bounded, integrable function.  		
	The problem has a variational structure and we prove the existence of a solution by using the classical Mountain-Pass Theorem. We work here with the harmonic extension of the fractional Laplacian, which allows us to deal with a weighted (but possibly degenerate) local operator, rather than with a nonlocal energy. \textcolor{black}{In order} to overcome the loss of compactness induced by the critical power we use a Concentration-Compactness principle. Moreover, a finer analysis of the geometry of the energy functional is needed in this convex case with respect to the concave-convex case studied in \cite{maria}.
	\end{abstract}
\tableofcontents

\section{Introduction and main statement}\label{intr}

\noindent  The goal of this paper is to prove the existence of a positive solution to the convex problem 
\begin{equation} \label{problem}(-\Delta)^s u =\varepsilon h u^q+u^{\cx-1} \qquad\mbox{ in }\mathbb{R}^n,
\end{equation}
where \textcolor{black}{$s\in (0,1)$, $n>2s$, $1\leq q <2^*_s-1$ are given quantities}, $\varepsilon>0$ is a small parameter, and $h$ is a function satisfying suitable summability conditions. The main operator in this problem is the fractional Laplacian, defined by
\[ (-\Delta)^s u(x)  = C(n,s) \mbox{ P.V. } \int_{\Rn} \frac{u(x)-u(y)}{|x-y|^{n+2s}}\, dy\]
	for any $x\in \Rn$ \textcolor{black}{and for a function $u$ regular enough\footnote{It is enough to take $u\in\mathcal{S}(\Rn)$ (the Schwartz space of rapidly decreasing functions), or in $ L^{\infty}(\Rn)$ and $C^{2s+\epsilon}$ (for some small $\epsilon>0$) in a neighborhood of $x$, to have a pointwise definition of the fractional Laplacian. Check also \cite{silvestre} for a refinement of the space of definition.}}, where $C(n,s)$ is a positive constant. For details on this operator and applications, see \cite{nonlocal}. See also \cite{hitch} for an introduction to fractional Sobolev spaces. 
\medskip

\noindent \textcolor{black}{The main result of this paper goes as follows.
	\begin{teo}\label{thm}
	Let $q\in [1, \cx-1)$ and $h$ be such that
	\eqlab{\label{h1} & h\in L^1(\Rn)\cap L^{\infty}(\Rn)\quad \mbox{and}\\
	& \mbox{there exists a ball }B\subset \Rn \mbox{ such that } \inf_B h>0.\\
	& \mbox{If }n\in (2s,6s), \mbox{suppose in addition }h\geq 0.}
	Let $\eps>0$ be a small parameter. Then problem \eqref{problem}
 admits a positive (mountain-pass) solution, provided that
	 $n>{\frac{2s(q+3)}{q+1} }$.
	\end{teo}}	
		
The literature concerning problems with this type of nonlinearities is large and deep in the classical case, see for instance \cite{ABC, AGP1, AGP, ALM, CW, Cing}, among others. In particular, in \cite{AGP} A. Ambrosetti, J. Garc\'ia-Azorero and I. Peral studied \eqref{problem} for $s=1$. There, \textcolor{black}{the existence of solutions} is proved by means of two different techniques: bifurcation and concentration-compactness. In the first case, they construct solutions for the whole range $0<q<2^*_s-1$ as small perturbations of the solutions to the problem
$$-\Delta u = u^{2^*-1}\qquad \mbox{ in }\mathbb{R}^n,\qquad u>0,$$
by using a Lyapunov-Schmidt reduction. \textcolor{black}{
On the other hand, the authors also prove the existence of two solutions for $0<q<1$ (that is, the concave-convex problem) by applying an argument of concentration-compactness type (in the spirit of \cite{lions1, lions2}).}

The fractional counterpart of these results is as follows. In \cite{DMPV}, a solution to \eqref{problem} for $0<q<2^*_s-1$ is obtained by means of a Lyapunov-Schmidt reduction. Indeed, the authors prove the existence of a function $w_\varepsilon$ (which goes to zero \textcolor{black}{in a suitable space with $\varepsilon\to 0$}) so that, for some $\mu\in (0,+\infty)$ and $\xi\in \mathbb{R}^n$, $z_{\mu,\xi}+w_\varepsilon$ solves the problem, where
\begin{equation} \label{zmu} z_{\mu,\xi}(x)=\mu^{\frac{2s-n}{2}}z\left(\frac{x-\xi}{\mu}\right),\qquad z(x)=\frac{c_*}{(1+|x|^2)^{\frac{n-2s}{2}}}\end{equation}
is a solution of
$$(-\Delta)^s u = u^{2^*_s-1}\qquad \mbox{ in }\mathbb{R}^n,\qquad u>0.$$
\textcolor{black}{Moreover, in \cite{maria} for the range $0<q<1$ the authors use the concentration-compactness principle to prove the existence of two solutions for problem \eqref{problem}}
    (see also \cite{bego, SV-2, SV-1} for related problems in the nonlocal case).
\medskip

\textcolor{black}{In this paper, we solve problem \eqref{problem}
in the fractional case $s\in (0,1)$ and in the range $1\leq q<2^*_s-1$, using a concentration-compactness principle. Notice that in our problem the two nonlinearities are convex, and the geometry of the functional suggests the existence of one solution instead of two. In order to prove the existence of a solution we use, roughly speaking, the following strategy:}

(i) To consider the energy functional associated to \eqref{problem} and to prove that it satisfies some compactness condition (Palais-Smale condition) under a certain energy level.

(ii) To build a sequence of functions with an appropriate geometry (of Mountain Pass type) whose energy lies below the critical level found in (i). 

(iii) To apply the Mountain Pass Lemma (see \cite{mps}) to pass to the limit, getting a solution.

There are two fundamental points here: to identify the energy level, and to find the appropriate sequence. \textcolor{black}{We point out that, in the concave-convex (fractional) problem studied in \cite{maria}}, the geometry derived from the concave term (the functional has a minimum of negative energy) helps to prove that the sequence stays below the critical level. However, in our paper both nonlinearities are convex, and the proof gets more involved. Indeed, if one adapts straightforwardly the compactness result in \cite[Proposition 4.2.1]{maria} and builds the sequence in the standard way (by considering the path along the Sobolev minimizers), then the arguments to prove that the energy of the sequence is small enough do not work. 

Thus, the study of \eqref{problem} will first require a finer analysis of the compactness properties of the functional. More precisely, we will have to improve the estimates of the functional in order to get a slightly higher critical level.
Accordingly, once we have found this new critical level, we perform a more careful analysis of the energy of the sequence given by the minimizers. We will finally conclude by applying the Mountain Pass Lemma in the standard way.

\textcolor{black}{We remark here that in this paper we also overcome 
a flaw found in \cite{AGP}, where the classical problem is studied}; indeed, to prove compactness (Proposition 2.1 therein) they state that the critical energy level $c_\varepsilon$ has to satisfy
$$c_\varepsilon<\frac{1}{n}S^{{n}/2}-C\varepsilon^{\frac{2^*}{2^*-(q+1)}}.$$
Nevertheless, if one follows the proof it arises that, in order to reach the contradiction, it has to be required that
$$c_\varepsilon<\frac{1}{{n}}S^{{n}/2}-C\varepsilon^{\frac{2^*}{2^*-(q+1)}}-C\varepsilon.$$
Notice that what we are saying here is that the compactness holds below a lower critical level, and thus it will be more difficult to find the sequence in (ii). This flaw was already fixed in \cite{maria} in the fractional, \textcolor{black}{concave-convex} case (see Proposition 4.2.1), where the authors consider the lower level and find the appropriate sequence. 

\bigskip

\textcolor{black}{We make now some preliminary observations on the problem that we study.} We see at first that if $h$ satisfies conditions \eqref{h1}, then also 
		\bgs{\label{h2} h\in L^m(\Rn) \mbox{ for any } m \in (1,+\infty).}
Furthermore, in the case $n\in (2s,6s)$ we need to ask $h$ to be positive. This restriction arises again from the study of the energy of the Sobolev minimizers. As we commented before, we would like to control the energy of the sequence that we will construct (and that will be based on the functions $z_{\mu,\xi}$, \textcolor{black}{see \eqref{zmu}}), and thus we would like the negative terms to be as large as possible. In particular, if one looks at the $q$-order term, we hope that the part where $h$ is positive \textcolor{black}{dominates over} the part where it is negative. To have this, we will center the function $z_{\mu,\xi}$ in the ball where $h$ is positive, so that the mass is concentrated there. However, it can be easily seen that for low dimensions the mass of the tails of the minimizers is too large and it annihilates the mass in the positive part of $h$. This computation gives an idea of why the necessity of \textcolor{black}{requiring} $h\geq 0$ for $n\in (2s,6s)$, but the detailed restriction can be found in Section \ref{proofThm}.
\medskip

The paper is organized as follows: in Section \ref{strategy} we provide the functional framework that will be needed, as well as some auxiliar results related to compactness and geometry properties. Section \ref{PSC} is devoted to the proof of the Palais-Smale condition for the energy functional, and Section \ref{minmax} to construct the sequence with mountain pass geometry and whose energy level lies below the critical one. Finally, in Section \ref{proofThm} we prove Theorem \ref{thm}.

	\section{Functional framework and preliminary computations}\label{strategy}

\noindent \textcolor{black}{We introduce at first some notations.} Let us denote by $\Rp:=\Rn\times (0,+\infty)$ the $n+1$ dimensional half-space, by $X=(x,y)\in \Rp$ a $n+1$ dimensional vector, having $x\in \Rn$ and $y>0$, and take $a:=1-2s$.  Moreover, for $x\in \Rn$ and $r>0$ we write $B_r(x)$ \textcolor{black}{(shorted to $B_r$ when $x=0$)} for the ball in $\Rn$ centered at $x$ with radius $r$, i.e.
	  \[ B_r(x):= \{ x'\in \Rn \mbox{ s.t. } |x-x'|<r\},\]
	  and for $X\in \Rp$ and $r>0$ we write $B_r^+(X)$ for the ball in $\Rp$ centered at $X$ with radius $r$, that is
	  \[    B_r^+(X):= \{ X'\in \Rp \mbox{ s.t. } |X-X'|<r\}.\]
Let us introduce first the seminorm
		\[ [u]^2_{\dot H^s(\Rn)}: = \iint_{\R^{2n}} \frac{|u(x)-u(y)|^2}{|x-y|^{n+2s}}\, dx\, dy,\] 
and define  the space $\dot H^s(\Rn)$ as the completion of the Schwartz space of rapidly decreasing smooth functions, with respect to the norm $[\,\cdot\,]_{\dot H^s(\Rn)} + \|\cdot\|_{L^{2^*_s}(\mathbb{R}^n)}$. 
	
\begin{defn}
We say that $u\in \dot H^s(\Rn)$ is a (weak) solution of $(-\Delta)^s u=f$ in $\Rn$ for a given $f\in L^\beta(\Rn)$ where $\beta:=2n/(n+2s)$ if
	\[ \frac{C(n,s)}{2} \iint_{\R^{2n}} \frac{ \left(u(x)-u(y)\right)\left(\varphi(x)-\varphi(y)\right)}{|x-y|^{n+2s}} \, dx \, dy = \int_{\Rn} f(x)\varphi(x)\, dx\] for every $\varphi \in \dot H^s(\Rn)$.
\end{defn}

\noindent Nevertheless, instead of directly working in this framework, we will transform the problem into a local one by using the extension due to L. Caffarelli and L. Silvestre (see \cite{extension}). 

Thus, the operator $(-\Delta)^s$ can be obtained as the trace of a local (possibly singular and degenerate) operator acting on the half space.  Given $U\colon \Rp\to\R$ that satisfies
		\syslab{ \label{extprc}&\mbox{div} (y^a \nabla U)= 0 & &\mbox{ in } \Rp,\\
			&	U(x,0)=u(x) &&\mbox{ in } \Rn,}
			it holds that, up to constants,
			\[ (-\Delta)^s u (x)= -\lim_{y \to 0^+} y^a \partial_y U(x,y).\]
			Let \[ [U]_a^*:=\left( \kappa_s\int_{\R^{n+1}} y^a |\nabla U|^2 \, dX\right)^{1/2} ,\] where $\kappa_s$ is a normalization constant. We then define the spaces
			\[ \dot H_a^s (\R^{n+1}) := \overline{ C_0^\infty(\R^{n+1})}^{[\cdot]_a^*}\]
			and
			\[ \Hsa := \Big\{ U:=\tilde U\Big|_{\Rp} \mbox{ s.t. }\tilde U \in \dot H^s_a (\R^{n+1}), \tilde U(x,y)=\tilde U(x,-y) \mbox{ a.e. in } \Rn \times \R \Big\}.\]
The norm in $\Hsa$ is, neglecting the constants,
	\bgs{\label{norm} [U]_a:= \lr{\int_{\R^{n+1}_+} y^a|\nabla U|^2 \, dX }^{1/2}.} 
	        
\noindent So, finding a solution $u\in \dot H^s(\Rn)$ of the nonlocal problem $(-\Delta)^s u=f(u)$ is equivalent to finding a solution $U\in \Hsa$ of the local problem
\sys{&\mbox{div} \left(y^a\nabla U\right) =0 && \mbox{ in } \Rp,\\
& -\lim_{y\to 0^+} y^a \partial_y U = f(u) && \mbox{ in } \Rn .}
Since we are looking for positive solutions of \eqref{problem}, we will consider the problem
\eqlab{\label{eqpls} (-\Delta)^s u =\varepsilon h u _+^q + u_+^{\cx-1} \quad \mbox{ in } \Rn}
and (according to the considerations above) its equivalent formulation 
\syslab{\label{equiveqplus} &\mbox{div} \left(y^a\nabla U\right) =0 && \mbox{ in } \Rp,\\
& -\lim_{y\to 0^+} y^a \partial_y U (x,y)= \varepsilon h U _+^q(x,0) + U_+^{\cx-1}(x,0) && \mbox{ in } \Rn.}
In particular, we say that $U\in \Hsa$ is a (weak) solution on the problem \eqref{equiveqplus} if
\bgs{\label{weaksolU} \int_{\Rp} y^a\langle \nabla U,\nabla \varphi\rangle \, dX =\int_{\Rn} \lr{ \varepsilon h(x) U_+^q(x,0) +U_+^{\cx-1}(x,0)} \varphi(x,0)\, dx,}
for every $\varphi \in \Hsa$. Furthermore, the energy functional associated to the problem \eqref{equiveqplus} is 
\bgs{ \label{oper} \Fl (U):=\frac{1}2 \int_{\Rp}y^a |\nabla U|^2 \, dX - \frac{\varepsilon}{q+1} \int_{\Rn} h(x) U_+^{q+1} (x,0)\, dx -\frac{1}{\cx}\int_{\Rn} U_+^{\cx}(x,0)\, dx.} 
In particular $\Fl \in C^1(\Hsa)$ and for any $U,V\in \Hsa$
\bgs{\label{operderiv}&  \langle \Fl'(U),V\rangle\\ =&\; \int_{\Rp} y^a\langle \nabla U,\nabla V\rangle \, dX - \varepsilon \int_{\Rn} h(x) U_+^q (x,0)V(x,0) - \int_{\Rn} U_+^{\cx-1}(x,0) V(x,0)\, dx.}
The purpose of the paper from here on is to prove the existence of a critical point $U$ of the operator $\Fl$. Then, $U$ is a solution of \eqref{equiveqplus} and therefore $u:=U(\cdot, 0)$ is a solution of \eqref{eqpls}. Moreover, one can prove that any nontrivial solution $u$ of \eqref{eqpls} (hence its extension $U$) is nonnegative, and therefore a true solution of \eqref{problem} (see for this \cite[Proposition 2.2.3]{maria}).
\medskip

It is known that (up to constants) the harmonic extension of the fractional Laplacian gives an isometry between $\dot H^s(\Rn)$ and $\Hsa$, i.e.
	\eqlab{\label{isom}    [u]_{\dot H^s(\Rn)} = [U]_a.}	We recall that the Sobolev embedding in $\dot H^s(\Rn)$ gives that
	\[ S\|u\|_{\Lst}^2 \leq [u]^2_{\dot H^s(\Rn)},\] where $S$ is the best constant of the Sobolev embedding of $\dot H^s(\Rn)$ (see for instance \cite[Theorem 6.5]{hitch}). As a consequence, we have the following inequality,
	\begin{prop}[Trace inequality] \label{traceIneq}
Let $U \in \Hsa$. Then
\bgs{ S\|U(\cdot, 0)\|_{\Lst}^{\textcolor{black}{2}}\leq [U]_a^2.}
\end{prop}			
\noindent In \cite[Theorem 1.1]{costi} the best Sobolev constant and the fractional Sobolev minimizers are explicitly computed. The form of the fractional Sobolev minimizer is given by
	\eqlab{\label{SobMin}  z(x):= \frac{c_\star}{(1+|x|^2)^{\frac{n-2s}{2}}}} for a positive constant $c_\star=c_\star(n,s)$. 

We introduce for $r\in(1,+\infty)$ the weighted Lebesgue space endowed with the norm
\[\|U\|_{L^r(\Rp,y^a)} :=\lr{\int_{\Rp}y^a |U|^r\, dX}^{1/r}.\] The following result gives a continuous Sobolev embedding of the space $\Hsa$ into the weighted Lebesgue space for a particular value of $r$. See for the proof \cite[Proposition 3.1.1]{maria}.

\begin{prop}[Sobolev embedding] \label{sob}There exists a constant $\widehat S>0$ such that for all $U\in \Hsa$ it holds that
\bgs{ \lr{ \int_{\Rp} y^a|U|^{2\gamma} \, dX}^{1/2\gamma} \leq \widehat S \lr{ \int_{\Rp} y^a |\nabla U|^2 \, dX}^{1/2},}
where $\gamma=1+2/(n-2s)$. 
\end{prop}

\noindent In the next proposition, we prove a useful integral inequality that will be frequently used.
\begin{prop} \label{qineq} Let $1\leq q <\cx-1$. Assume $u\in L^{2^*_s}(\mathbb{R}^n)$  and $h\in L^m(\mathbb{R}^n)$ with $m=\frac{2^*_s}{2^*_s-(q+1)}$. Then,
\bgs{\label{bound3} \left| \int_{\Rn} h(x) u^{q+1}(x) \, dx\right| \leq \|h\|_{L^m(\Rn)} \|u\|^{q+1}.}
\end{prop}

\begin{proof} We use the Hölder inequality to deduce that
	\bgs{ \left| \int_{\Rn} h(x) u^{q+1} (x) \, dx \right| \leq \al \int_{\Rn} |h(x)| u^{q+1} (x) \, dx\\
	\leq \al   \lrq{\int_{\Rn}|h|^{\frac{\cx}{\cx-q-1} } \, dx }^{\frac{\cx-q-1}{\cx}} \lrq{ \int_{\Rn} |u| ^{\cx} \, dx }^{\frac{q+1}{\cx}}\\
	\leq \al \|h\|_{L^m(\Rn)} \|u\|_{\Lst}^{q+1}  } for $m= \frac{\cx}{\cx-q-1}>1$, and so the inequality is proved. 
\end{proof}

\noindent The next proposition is the equivalent of \cite[Lemma 4.1.1]{maria} in the case $q\geq1$ and goes as follows.
\begin{prop}\label{convres}
Let $v_k \in L^{\cx}(\Rn,[0,+\infty))$ be a sequence converging to some $v$ in $L^{\cx}(\Rn)$. Then for any $r>1$
\[ \lim_{k \to +\infty} \int_{\Rn} |v_k^r(x) -v^r(x)|^{\frac{\cx}{r}} \, dx =0.\]
\end{prop}

\begin{proof}
For any $a\geq b\geq 0$ and any $r>1$ we see that
\[ a^r-b^r =r\int_{b}^{a} t^{r-1}\, dt \leq r a^{r-1} (a-b) \leq r(a^{r-1}+b^{r-1})(a-b).\]
Exchanging $a$ with $b$, we conclude that
\[ \big| a^r-b^r \big|\leq r (a+b)^{r-1}|a-b|.\]
Then by the Hölder inequality we have that
\bgs{\int_{\Rn} |v_k^r(x)-v^r(x)|^{\frac{\cx}{r}}  \, dx \leq &\; r^{\frac{\cx}{r}} \lr{\int_{\Rn} (v_k+v)^{\cx }\, dx} ^{(r-1)/r} \lr{ \int_{\Rn} |v_k-v|^{\cx}\, dx}^{1/r} \\
\leq &\; r^{\frac{\cx}{r}} \
\|v_k+v\|_{\Lst}^{\frac{(r-1) \cx}{r}} \|v_k-v\|_{\Lst}^{\frac{\cx}{r}}   .} 
Using the convergence $\|v_k-v\|_{\Lst}\to 0$ (from which it also follows that $\|v_k+v\|_{\Lst}$ is uniformly bounded), the conclusion plainly follows.
\end{proof}

\noindent Another useful result is given in \cite[Lemma 4.2.4]{maria}. We just notice that now, for $q>1$, the statement goes as follows:
\begin{prop}\label{proppqs}
Let $m:= \frac{\cx}{\cx-(q+1)}$. Then there exists a positive constant $\bar C$ depending on $n,s,q$ and $\|h\|_{L^m(\Rn)}$ such that, for any $\alpha>0$,
\bgs{ \frac{s}n \alpha^{\cx} -\eps\lr{\frac{1}2-\frac{1}{q+1}}\|h\|_{L^m(\Rn)}\alpha^{q+1}\geq -\bar C\eps^{\frac{\cx}{\cx-(q+1)}}.}
\end{prop}

\section{Palais-Smale condition}\label{PSC}
\noindent The main result of this Section is the following.
\begin{teo}\label{PSCthm} 
There exists~$\bar C, c_1>0$, depending on~$h, q, n$
and~$s$, such that the following statement
holds true. 

Let $\{U_k\}_{k\in \N}\subset \Hsa$ be a sequence satisfying
\begin{enumerate}
\item[(i)]$\displaystyle\lim_{k\to+\infty}\mathcal{F}_\varepsilon(U_k)= 
c_\varepsilon$, with 
\begin{equation*}\begin{split}\label{ceps}
&c_\varepsilon+c_1\varepsilon^{1+\delta} +\overline C \eps^{\frac{\cx}{\cx-(q+1)}}<
\dfrac{s}{n}S^{\frac{n}{2s}}\qquad  \hbox{ if }n\geq 6s,\\
&c_\varepsilon+c_1\varepsilon^{1+\delta}<
\dfrac{s}{n}S^{\frac{n}{2s}} \qquad \hbox{ if }n\in (2s,6s),
\end{split}\end{equation*} 
where $\delta>0$ and $S$ is the Sobolev constant appearing in Proposition~\ref{traceIneq},
\item[(ii)]$\displaystyle\lim_{k\to+\infty}\mathcal{F}'_\varepsilon(U_k)= 0.$
\end{enumerate}
Then there exists a subsequence, still denoted by~$\{U_k\}_{k\in\mathbb{N}}$, 
which is strongly convergent in $\dot{H}^s_a(\mathbb{R}^{n+1}_+)$ as~$k\to+\infty$.
\end{teo}

\noindent \textcolor{black}{Here, the limit in $(ii)$ is to be intended as
\[ \lim_{k\to +\infty}  \|\mathcal F'(U_k) \|_{\mathcal L(E,E)} = \lim_{k\to +\infty} \sup_{V\in E,\|V\|_{E}=1} \left|\scp{ \mathcal F'(U_k),V }\right| = 0, \] where we denote by $\mathcal L(E,E)$ the space of all linear functionals from $E$ to $E$.}
\begin{oss}
As we commented in the introduction, one of the key points in this work is to slightly improve the critical level in such a way that further on we can build a sequence whose energy lies below it. This is precisely the role played by the parameter $\delta$ in the previous theorem. We can not drop this term (that will cause important difficulties) but we can choose $\delta$ large enough so that we can neglect it when $\varepsilon\rightarrow 0$.
\end{oss}

\noindent We recall at first a concentration-compactness principle, stated in \cite[Proposition 3.2.3]{maria} and proved there. This principle is based on the original results by P.L Lions in \cite{lions1,lions2} (in particular  in \cite[Lemma 2.3]{lions2}). For this, we recall the next definitions.
\begin{defn}
A sequence $\{U_k\}_{k\in \N}$ is tight if for every $\mu>0$ there exists $\rho>0$ such that for any $k\in \N$
\[ \int_{\Rp\setminus B_{\rho}^+}y^a |\nabla U_k|^2 \, dX\leq \mu.\]
\end{defn}
\begin{defn}\label{measconv}
Let $\{\mu_k\}_{k\in \N}$ be a sequence of measures on a topological space $X$. We say that $\mu_k $ converges to $\mu$ on $X$ if and only if
\[\lim_{k\to +\infty} \int_X \varphi d\mu_k =\int_X \varphi \, d\mu \quad \mbox{ for any } \varphi \in C_0(X).\]
\end{defn}
\noindent Then the principle goes as follows.
\begin{prop}[Concentration-Compactness Principle]\label{CCP}
Let $\{U_k\}_{k\in \N}$ be a bounded and tight sequence in $\Hsa$ such that $U_k$ converges weakly to $U$ in $\Hsa$. Let $\mu,\nu$ be two nonnegative measures on $\Rp$ respectively $\Rn$ such that (in the sense of Definition \ref{measconv})
\bgs{\label{measconv1} \lim_{k \to +\infty} y^a|\nabla U_k|^2 =\mu }
and
\bgs{\label{measconv2} \lim_{k \to +\infty} | U_k(x,0)|^{2_s^*} =\nu. }
 Then there exists a set $J$ that is at most countable and three families $\{x_j\}_{j\in J}\in \Rn$, $\{\nu_j\}_{j\in J}$ and $\{\mu_j\}_{j\in J}$ with $\nu_j, \mu_j\geq0$ such that
\begin{flalign*}
\mbox{ (i) } &\nu=|U(x,0)|^{2_s^*}+\sum_{j\in J} \nu_j \delta_{x_j} \qquad \qquad \qquad\qquad\qquad\qquad\qquad\qquad\qquad\qquad\qquad\qquad\qquad\qquad\qquad \\
\mbox{ (ii) }& \mu \geq y^a|\nabla U|^2 + \sum_{j\in J}\mu_j \delta_{\{x_j,0\}}\\
\mbox{ (iii) }&\mu_j\geq S\nu_j^{2/2_s^*} \mbox{ for all } j \in J.
\end{flalign*}

\end{prop}
\noindent We prove that a sequence $\{U_k\}_{k\in\N}$ satisfying the assumptions in Theorem \ref{PSCthm} is bounded. \textcolor{black}{A slighter more general result is given in the following Lemma.}

\begin{lem}\label{Ukisbd}
Let $\varepsilon, \kappa>0$ and let $\{U_k\}_{k\in \N} \subset \Hsa$ be a sequence that satisfies
\eqlab{ \label{bound1} |\Fl (U_k)| + \sup_{V\in \Hsa, \,  [V]=1}| \scp{\Fl' (U_k),V}|\leq \kappa }
for any $k\in \N$. Then there exists $M>0$ such that for any $k\in \N$
	\eqlab{\label{bound}  [U_k]_a\leq M.}
\end{lem}

\begin{proof}
We suppose by contradiction that for every $M>0$ there exists $k\in \N$ such that
	\eqlab{ \label{ra}  [U_k]_a>M.}
Thanks to \eqref{bound1} we have that
	\bgs{  \kappa \geq \al \Fl (U_k) = \frac{1}{2} [U_k]_a^2 -\frac{\varepsilon}{q+1} \int_{\Rn} h(x) (U_k)_+ ^{q+1}(x,0)\, dx -\frac{1}{\cx}\int_{\Rn} (U_k)_+^{\cx} (x,0)  \, dx.}
Using also the bound in \eqref{bound3}, we obtain that
	\eqlab{ \label{e1}   \,[U_k]_a^2 \leq \al 2 \kappa + \frac{2\varepsilon}{q+1} \int_{\Rn} h(x) (U_k)_+ ^{q+1}(x,0)\, dx +\frac{2}{\cx}\int_{\Rn} (U_k)_+^{\cx} (x,0) \, dx \\
		\leq \al 2 \kappa + \frac{2 \varepsilon}{q+1} \|h\|_{L^m(\Rn)} \|(U_k)_+\|_{\Lst}^{q+1} +\frac{2}{\cx}\| (U_k)_+\|_{\Lst}^{\cx} .}
Thus, from this and \eqref{ra}, we deduce that also \textcolor{black}{for every $\tilde{M}>0$ one can find $k\in\mathbb{N}$ so that}
	\eqlab{ \label{ra2} \|\ukp\|_{\Lst}>\tilde{M}.} 
Consider now the function $f\colon (0,\infty) \to (0,\infty)$ defined as 
	\[ f(\tau):=\frac{\tau^{q+1}}{\tau^{\cx}}.\] 
Since $q+1<\cx$ we have that
	\[\lim_{\tau \to \infty}f(\tau)=0,\] 
and hence, for any $\delta>0$ there exists $\tau_\delta>0$ such that for every $\tau>\tau_\delta$, one has that $f(\tau)< \delta$. Hence, fixing  $0<\delta<1$, by \eqref{ra2} we can assume
\begin{equation}\label{lowBound1}
\|(U_k)_+\|_{\Lst}>\tau, \qquad \|\ukp\|_{\Lst}^{q+1}\leq \delta\|\ukp\|_{\Lst}^{\cx},\qquad \forall \; \tau>\tau_\delta.
\end{equation}	
Therefore, by Proposition \ref{traceIneq} \textcolor{black}{there exists $k\in \N$ such that }
\begin{equation}\label{lowBound2}
[U_k]_a>\tau S^{1/2}, \qquad \forall \; \tau>\tau_\delta.
\end{equation}
	\textcolor{black}{Using \eqref{lowBound1} and} \eqref{e1} we obtain that
	\eqlab{ \label{bound2}   	[U_k]_a^2  \leq 2\kappa+ \lr{\delta \frac{2\varepsilon}{q+1} \|h\|_{L^m(\Rn)}+ \frac{2}{\cx}  } \|\ukp\|_{\Lst}^{\cx}.}
	
\noindent On the other hand, considering the quotient $U_k / [U_k]_a$ from \eqref{bound1} we get that
	\[ |\scp{\Fl'(U_k),U_k}|\leq \kappa [U_k]_a.\]
	From this and the fact that $|\Fl(U_k)|\leq \kappa$, for $q>1$ we have that
		\eqlab{\label{usel} \kappa (1+[U_k]_a)\geq \al \Fl(U_k) - \frac{1}{2}\scp{\Fl'(U_k),U_k} \\ =\al \varepsilon\lr{\frac{1}{2} -\frac{1}{q+1} }\int_{\Rn} h(x) \ukp^{q+1}(x,0)\, dx + \frac{s}{n} \|\ukp \|_{\Lst}^{\cx},}
		\textcolor{black}{recalling that
		\[ \frac{1}2-\frac{1}{\cx} =\frac{s}n.\]}
Thanks to the bound in \eqref{bound3}, it follows that
	\bgs{ \frac{s}{n} \|\ukp\|_{\Lst}^{\cx} \leq\al   \kappa (1+[U_k]_a) + \varepsilon\lr{\frac{1}{2} -\frac{1}{q+1} } \|h\|_{L^m(\Rn)} \|\ukp\|_{\Lst}^{q+1}.}
We use \eqref{lowBound1} again and we obtain that 
			\bgs{ \frac{s}{n} \|\ukp\|_{\Lst}^{\cx} \leq\al \kappa (1+[U_k]_a) +\delta \varepsilon \lr{\frac{1}{2} -\frac{1}{q+1} } \|h\|_{L^m(\Rn)}  \|\ukp\|_{\Lst}^{\cx}  .}
			Thus
					\bgs{\lrq{ \frac{s}{n}-  \delta \varepsilon \lr{\frac{1}{2} -\frac{1}{q+1} } \|h\|_{L^m(\Rn)}  } \|\ukp\|_{\Lst}^{\cx} \leq\al \kappa (1+[U_k]_a)  ,}
					which for $\delta$ small enough, implies that
						\[ c \|\ukp\|_{\Lst}^{\cx}\leq \kappa(1+[U_k]_a).\] 
Notice that for $q=1$ the inequality above immediately follows from \eqref{usel}. 
						This, together with \eqref{bound2}, yields 
						\[ [U_k]_a^2\leq C_1+ C_2 [U_k]_a \] 
						for suitable positive constants $C_1, C_2$, both independent of $k$. Choosing $\tau$ large enough in \eqref{lowBound2} we contradict this inequality and conclude the proof.
\end{proof}
\noindent Furthermore, a sequence $\{U_k\}_{k\in \N} \subset \Hsa$ that satisfies the hypotheses of Theorem \ref{PSCthm} is tight, as stated in the next Lemma \ref{tightness}. We follow the steps of the proof of Lemma 4.2.5  in \cite{maria} with the needed adjustments, and give here all the details for the sake of completeness.

\begin{lem}\label{tightness} Let $\{U_k\}_{k\in \N} \subset \Hsa$ be a sequence that satisfies the hypothesis of Theorem \ref{PSCthm}. Then for any $\eta>0$ there exists $\rho>0$ such that for any $k\in \N$ it holds that
\[ \int_{\Rp\setminus B_{\rho}^+} y^a|\nabla U_k|^2 \, dX + \int_{\Rn \setminus \{B_\rho \cap\{y=0\}\}} (U_k)^{\cx} (x,0)\, dx<\eta.\] In particular, the sequence $\{U_k\}_{k\in \N}$ is tight.
\end{lem}
\begin{proof}
First we notice that~\eqref{bound1}
holds in this case, due to conditions~(i) and~(ii) in
Theorem~\ref{PSCthm}. Hence,
Lemma \ref{Ukisbd} 
gives that the sequence 
$\{U_k\}_{k\in\mathbb{N}}$ is uniformly bounded in $\dot{H}^s_a(\mathbb{R}^{n+1}_+)$, and thus
\begin{equation}\begin{split}\label{weak convergence-1}
& U_k\rightharpoonup U \quad \hbox{ in }\dot{H}^s_a(\mathbb{R}^{n+1}_+) \quad {\mbox{ as }}k\to+\infty \\
{\mbox{and }}& U_k\rightarrow U\quad  \hbox{ a.e. in }\mathbb{R}^{n+1}_+\quad {\mbox{ as }}k\to+\infty.
\end{split}\end{equation}

\noindent We now proceed by contradiction. Suppose that there exists $\eta_0>0$ 
such that for all $\rho>0$ there exists~$k=k(\rho)\in\N$ such that
\begin{equation}\label{contrad}
\int_{\mathbb{R}^{n+1}_+\setminus B_\rho^+}{y^a|\nabla U_k|^2\,dX}
+\int_{\mathbb{R}^n\setminus\{B_\rho\cap\{y=0\}\}}{(U_k)_+^{2^*_s}(x,0)\,dx}
\geq \eta_0.
\end{equation}
We observe that 
\begin{equation}\label{forse0} 
k\to+\infty \quad {\mbox{ as }}\rho\to+\infty.
\end{equation}
Indeed, let us take a sequence $\{\rho_i\}_{i\in\mathbb{N}}$ such that $\rho_i\rightarrow +\infty$ as $i\rightarrow +\infty$, and suppose that $k_i:=k(\rho_i)$ given by \eqref{contrad} is a bounded sequence. That is, the set $F:=\{k_i:\;i\in\N\}$ is a finite set of integers.

Hence, there exists an integer $k^\star$ so that we can extract a subsequence $\{k_{i_j}\}_{j\in\N}$ satisfying $k_{i_j}=k^\star$ for any $j\in\N$. Therefore,
\begin{equation}\label{contrad2}
\int_{\mathbb{R}^{n+1}_+\setminus B_{\rho_{i_j}}^+}{y^a|\nabla U_{k^\star}|^2\,dX}
+\int_{\mathbb{R}^n\setminus\{B_{\rho_{i_j}}\cap\{y=0\}\}}{(U_{k^\star})_+^{2^*_s}(x,0)\,dx}
\geq \eta_0,
\end{equation}
for any $j\in \N$. 
But on the other hand, since~$U_{k^\star}$
belongs to~$\dot{H}^s_a(\mathbb{R}^{n+1}_+)$ 
(and so $U_{k^\star}(\cdot,0)\in L^{2^*_s}(\R^n)$ 
thanks to Proposition \ref{traceIneq}), for $j$ large enough there holds
\begin{equation*}
\int_{\mathbb{R}^{n+1}_+\setminus B_{\rho_{i_j}}^+}{y^a|\nabla U_{k^\star}|^2\,dX}
+\int_{\mathbb{R}^n\setminus\{B_{\rho_{i_j}}\cap\{y=0\}\}}{(U_{k^\star})_+^{2^*_s}(x,0)\,dx}
\leq \frac{\eta_0}{2},
\end{equation*}
which is a contradiction with \eqref{contrad2}.
This shows~\eqref{forse0}. 

Now, since $U$ given in~\eqref{weak convergence-1} belongs to~$\in\dot{H}^s_a(\mathbb{R}^{n+1}_+)$, by Propositions~\ref{traceIneq} 
and~\ref{sob} 
we have that for a fixed $\varepsilon>0$, there exists $r_\varepsilon>0$ such that
$$\int_{\mathbb{R}^{n+1}_+\setminus B_{r_\varepsilon}^+}{y^a|\nabla U|^2\,dX}
+\int_{\mathbb{R}^{n+1}_+\setminus B_{r_\varepsilon}^+}{y^a|U|^{2\gamma}\,dX}
+\int_{\mathbb{R}^n\setminus\{B_{r_\varepsilon}\cap\{y=0\}\}}{|U(x,0)|^{2^*_s}\,dx}<\varepsilon^\alpha,$$
with $\alpha>\gamma$ 
 and $\gamma$ defined in Proposition \ref{sob}. Notice that, without loss of generality, we can assume that 
\begin{equation}\label{eps to zero}
{\mbox{$r_\eps\to +\infty$ as $\eps\to 0$.}} 
\end{equation}
On the other hand, since $h\in L^m(\R^n)$ for every $m\in(1,+\infty)$, in particular we can assure the existence of a radius $\bar{r}_\eps$ such that
\begin{equation}\label{heps}
\|h\|_{L^{m}(\R^n\setminus B_{\bar{r}_\eps})}\leq \eps^{\beta},
\end{equation}
with $m$ satisfying $\frac{1}{m}=1-\frac{q+1}{2^*_s}$ and $\beta>\alpha/\gamma-1$.\\
Moreover, by \eqref{bound} and again by Propositions~\ref{traceIneq}
and~\ref{sob}, there exists $\tilde{M}>0$ such that
\begin{equation}\label{boundk}
\int_{\mathbb{R}^{n+1}_+}{y^a|\nabla U_k|^2\,dX}+\int_{\mathbb{R}^{n+1}_+}{y^a|U_k|^{2\gamma}\,dX}
+\int_{\mathbb{R}^n}{|U_k(x,0)|^{2^*_s}\,dx}\leq \tilde{M}.
\end{equation}
Let \textcolor{black}{\eqlab{\label{chooser} r:=\max\{r_\eps,\bar{r}_\eps\}.}} Now let $j_\eps\in\mathbb{N}$ be the integer part of $\frac{\tilde{M}}{\varepsilon^\alpha}$. Notice that~$j_\eps$ tends to~$+\infty$ as~$\eps$ 
tends to~0. We also set
$$ I_l:=\{(x,y)\in\mathbb{R}^{n+1}_+:r+l\leq |(x,y)|\leq r+(l+1)\},\;l=0,1,\cdots,j_\eps.$$
Thus, from~\eqref{boundk} we get
\begin{eqnarray*}
(j_\eps+1)\varepsilon^\alpha &\geq &
\frac{\tilde{M}}{\varepsilon^\alpha}\varepsilon^\alpha \\&\ge &
 \sum_{l=0}^{j_\eps}\left({\int_{I_l}{y^a|\nabla U_k|^2\,dX}
 +\int_{I_l}{y^a|U_k|^{2\gamma}\,dX}
+\int_{I_l\cap\{y=0\}}{|U_k(x,0)|^{2^*_s}\,dx}}\right),
\end{eqnarray*}
and this implies the existence of $\bar{l}\in\{0,1,\cdots, j_\eps\}$ such that, 
up to a subsequence, 
\begin{equation}\label{epsBound}
\int_{I_{\bar{l}}}{y^a|\nabla U_k|^2\,dX}+\int_{I_{\bar{l}}}{y^a|U_k|^{2\gamma}\,dX}
+\int_{I_{\bar{l}}\cap\{y=0\}}{|U_k(x,0)|^{2^*_s}\,dx}\leq \varepsilon^\alpha .
\end{equation}
We take now a cut-off function~$\chi\in C^\infty_0(\R^{n+1}_+,[0,1])$, 
such that
\begin{equation}\label{3.4bis}
\chi(x,y)=\begin{cases}
1,\quad |(x,y)|\leq r+\bar{l}\\
0,\quad |(x,y)|\geq r+(\bar{l}+1),
\end{cases}
\end{equation}
and 
\begin{equation}\label{3.4bisbis}
|\nabla \chi|\leq 2.
\end{equation} 
We also define 
\begin{equation}\label{3.4ter}
V_k:=\chi U_k \quad {\mbox{ and }}\quad W_k:=(1-\chi)U_k.
\end{equation}
We estimate
\begin{equation}\begin{split}\label{math F}
&|\langle \mathcal{F}'_\eps(U_k)-\mathcal{F}'_\eps(V_k),V_k\rangle |\\
&\quad = \bigg|\int_{\mathbb{R}^{n+1}_+}{y^a\langle\nabla U_k,\nabla V_k\rangle\,dX}
-\eps \int_{\mathbb{R}^{n}}{h(x) (U_k)_+^q(x,0)\,V_k(x,0)\,dx}\\
&\qquad\quad -\int_{\mathbb{R}^{n}}{(U_k)_+^{\cx-1}(x,0)\,V_k(x,0)\,dx}
-\int_{\mathbb{R}^{n+1}_+}{y^a\langle\nabla V_k,\nabla V_k\rangle\,dX}\\
&\qquad\quad +\eps \int_{\mathbb{R}^{n}}{h(x) (V_k)_+^{q+1}(x,0)\,dx}
+\int_{\mathbb{R}^{n}}{(V_k)_+^{\cx}(x,0)\,dx}\bigg|.
\end{split}\end{equation}
First, we observe that
\begin{equation}\begin{split}\label{AA}
&\bigg|\int_{\mathbb{R}^{n+1}_+}{y^a\langle\nabla U_k,\nabla V_k\rangle\,dX}-\int_{\mathbb{R}^{n+1}_+}{y^a\langle\nabla V_k,\nabla V_k\rangle\,dX}\bigg|\\
&\qquad\leq \int_{I_{\overline{l}}}{y^a|\nabla U_k|^2|\chi||1-\chi|\,dX}+\int_{I_{\overline{l}}}{y^a|\nabla U_k||U_k||\nabla\chi|\,dX}\\
&\qquad\qquad +2\int_{I_{\overline{l}}}{y^a|U_k||\nabla U_k||\nabla\chi||\chi|\,dX}+\int_{I_{\overline{l}}}{y^a|U_k|^2|\nabla \chi|^2\,dX}\\
&\qquad =:A_1+A_2+A_3+A_4.
\end{split}\end{equation}
By \eqref{epsBound}, we have that $A_1\leq C\varepsilon^\alpha $, for some $C>0$. 
Furthermore, by the H\"older inequality, \eqref{3.4bisbis} and \eqref{epsBound}, we obtain
\begin{eqnarray*}
A_2 &\leq& 2\int_{I_{\overline{l}}}{y^a|\nabla U_k||U_k|\,dX}\leq 2\left(\int_{I_{\overline{l}}}{y^a|\nabla U_k|^2\,dX}\right)^{1/2}\left(\int_{I_{\overline{l}}}{y^a| U_k|^2\,dX}\right)^{1/2}\\
&\leq& 2\varepsilon^{\alpha/2}  \left(\int_{I_{\overline{l}}}{y^a|U_k|^{2\gamma}\,dX}\right)^{1/{2\gamma}}
\left(\int_{I_{\overline{l}}}y^{a}\,dX \right)^{\frac{\gamma-1}{2\gamma}}.
\end{eqnarray*}
Since $a=(1-2s)>-1$,
the second integral is finite, 
and therefore, for $\varepsilon<1$,
\begin{equation*}
A_2\leq \tilde{C}\varepsilon^{\alpha/2}  \left(\int_{ I_{\overline{l} }}{y^a|U_k|^{2\gamma}\,
dX}\right)^{1/{2\gamma}}\leq C\varepsilon^{\alpha/2} \eps^{\alpha/2\gamma}\le C\eps^{\alpha/\gamma},
\end{equation*}
where \eqref{epsBound} was used again. 
In the same way, we get that $A_3\leq C\varepsilon^{\alpha/\gamma}$. Finally, 
\begin{equation*}
A_4\leq C\left(\int_{I_{ \overline{l} }}{y^a|U_k|^{2\gamma}\,dX}\right)^{1/{\gamma}}
\left(\int_{I_{\overline{l}}}{y^{a}\,dX}\right)^{\frac{\gamma-1}{\gamma}}\leq C\eps^{\alpha/\gamma}.
\end{equation*}
Using this information in \eqref{AA}, since $\alpha>\alpha/\gamma$ we obtain that 
$$ \bigg|\int_{\mathbb{R}^{n+1}_+}{y^a\langle\nabla U_k,\nabla V_k\rangle\,dX}
-\int_{\mathbb{R}^{n+1}_+}{y^a\langle\nabla V_k,\nabla V_k\rangle\,dX}\bigg|
\le C\eps^{\alpha/\gamma}, $$
up to renaming the constant $C$. \\
On the other hand by \eqref{3.4ter} and \eqref{epsBound}, 
\begin{eqnarray*}
\bigg|\int_{\mathbb{R}^n}{\Big( (U_k)_+^{\cx-1}(x,0)\,V_k(x,0)-(V_k)_+^{\cx}(x,0)\Big)\,dx}\bigg|
&\le &\int_{\mathbb{R}^n}{|1-\chi^{\cx-1}||\chi| |U_k(x,0)|^{\cx}\,dx}\\
&\leq& C\int_{I_{\overline{l}}\cap\{y=0\}}{|U_k(x,0)|^{2^*_s}\,dx}\leq C\eps^{\alpha}.
\end{eqnarray*}
In the same way, applying the H\"older inequality, one obtains
\begin{equation}\begin{split}\label{ealpha}
&\bigg|\varepsilon\int_{\mathbb{R}^n}{h(x)\,\left((U_k)_+^q(x,0)\,V_k(x,0)-(V_k)_+^{q+1}(x,0)\right)\,dx}\bigg|
\\&\qquad \le \varepsilon\int_{\mathbb{R}^n}{|h(x)|\,|1-\chi^q||\chi| |U_k(x,0)|^{q+1}\,dx}\\
&\qquad \leq C\, \varepsilon\|h\|_{L^\infty(\R^n)}\,
\int_{I_{\overline{l}}\cap\{y=0\}}{|U_k(x,0)|^{2^*_s}\,dx}\leq C\eps^{1+\alpha}.
\end{split}\end{equation}

\noindent All in all, plugging these observations in \eqref{math F}, we obtain that 
\begin{equation}\label{boundV}
|\langle \mathcal{F}'_\eps(U_k)-\mathcal{F}'_\eps(V_k),V_k\rangle|
\leq C\eps^{\alpha/\gamma}.
\end{equation}
Likewise, one can see that
\begin{equation}\label{boundW}
|\langle \mathcal{F}'_\eps(U_k)-\mathcal{F}'_\eps(W_k),W_k\rangle|
\leq C\eps^{\alpha/\gamma}.
\end{equation}

\noindent Now we claim that 
\begin{equation}\label{fprimeV}
|\langle \mathcal{F}'_\eps(V_k),V_k\rangle|\leq C\eps^{\alpha/\gamma}+o_k(1),
\end{equation}
where $o_k(1)$ denotes (here and in the rest of this paper)
a quantity that tends to 0 as $k$ tends to $+\infty$. 
For this, we first observe that 
\begin{equation}\label{bbbb}
[V_k]_a\le C\mbox{ and }[W_k]_a\leq C,
\end{equation}
for some $C>0$. Indeed, recalling \eqref{3.4ter} and using \eqref{3.4bis} 
and \eqref{3.4bisbis}, we have 
\begin{eqnarray*}
[V_k]_a^2 &=& \int_{\R^{n+1}_+}y^a|\nabla V_k|^2\,dX \\ 
&=& \int_{\R^{n+1}_+}y^a|\nabla\chi|^2|U_k|^2\,dX + 
\int_{\R^{n+1}_+}y^a\,\chi^2|\nabla U_k|^2\,dX + 2\int_{\R^{n+1}_+}y^a\,\chi\,U_k
\ \langle \nabla U_k, \nabla\chi\rangle\,dX\\
&\le & 4 \int_{I_{\overline{l}} }y^a| U_k|^2\,dX + [U_k]_a^2 +
C\left(\int_{ I_{\overline{l}}}y^a|\nabla U_k|^2\,dX\right)^{1/2}\, 
\left(\int_{I_{\overline{l}}}y^a |U_k|^2\,dX\right)^{1/2}\\
&\le & C \left(\int_{I_{\overline{l}} }y^a| U_k|^{2\gamma}\,dX\right)^{1/\gamma} 
+ [U_k]_a^2 +
C\, [U_k]_a\, \left(\int_{I_{\overline{l}}}y^a |U_k|^{2\gamma}\,dX\right)^{1/2\gamma}, 
\end{eqnarray*}
where the H\"older inequality was used in the last two lines.  
Hence, from Proposition \ref{sob} and using \eqref{bound}, we obtain \eqref{bbbb}. The estimate for $W_k$ can be proved analogously.

Now, we notice that 
\begin{eqnarray*}
|\langle \mathcal{F}'_\eps(V_k),V_k\rangle| \le 
|\langle \mathcal{F}'_\eps(V_k)-\mathcal{F}'_\eps(U_k),V_k\rangle| + 
|\langle \mathcal{F}'_\eps(U_k),V_k\rangle| \le  
C\,\eps^{\alpha/\gamma} +|\langle \mathcal{F}'_\eps(U_k),V_k\rangle|,
\end{eqnarray*}
thanks to \eqref{boundV}. 
Thus, from \eqref{bbbb} and assumption (ii) in Theorem \ref{PSCthm} 
we get the desired claim in \eqref{fprimeV}. 

Analogously (but making use of \eqref{boundW}), one can see that  
\begin{equation}\label{fprimeW}
|\langle \mathcal{F}'_\eps(W_k),W_k\rangle|\leq C\eps^{\alpha/\gamma}+o_k(1),
\end{equation}

We give now the proof of  Lemma \ref{tightness} for $n\geq 6s$. We notice here that the computations that follow are also true in small dimensions. However, for $n<6s$, it is not possible in this way to construct a path that lies below the needed critical level (one can check the hypothesis (i) in Theorem \ref{PSC}). Indeed, for $n\in(2s,6s)$ the additional hypothesis that $h\geq 0$ (given in Theorem \ref{thm}) is required, as we see further on.

\bigskip

So for $n\geq 6s$, we divide the proof in three main steps: 
we first show lower bounds for $\mathcal{F}_\eps(V_k)$ 
and $\mathcal{F}_\eps(W_k)$ (see Step 1 and Step 2, respectively), 
and then in Step 3 we obtain a lower bound for $\mathcal{F}_\eps(U_k)$, 
which will give a contradiction with the hypotheses on $\mathcal{F}_\eps$, 
and so the conclusion of Lemma \ref{tightness}.

\medskip

\noindent {\it Step 1: Lower bound for $\mathcal{F}_\eps(V_k)$.} 
Recalling  that 
$$ \frac12-\frac{1}{\cx}=\frac{s}{n}$$ we have by Proposition \ref{qineq} that
\begin{equation*}\begin{split}
\mathcal{F}_\eps(V_k)&-\frac{1}{2}\langle \mathcal{F}'_\eps(V_k), V_k\rangle 
= \left(\frac{1}{2}-\frac{1}{\cx}\right)\|(V_k)_+(\cdot,0)\|_{\Lst}^{\cx}\\
&\;\;+\eps\left(\frac{1}{2}-\frac{1}{q+1}\right)\int_{\R^n}h(x)(V_k)_+^{q+1}(x,0)\,dx\\
&\geq \frac{s}n \|(V_k)_+(\cdot,0)\|_{\Lst}^{\cx} - \eps\left(\frac{1}{2}-\frac{1}{q+1}\right)\|h\|_{L^m(\Rn)} \|(V_k)_+(\cdot,0)\|_{\Lst}^{q+1},
\end{split}\end{equation*}
and by Proposition \ref{proppqs} and \eqref{fprimeV} we get that
\begin{equation}\label{LowerBoundV}
\mathcal{F}_\varepsilon(V_k)\geq -C\eps^{\alpha/\gamma}- \overline C \eps^{\frac{\cx}{\cx-(q+1)}}+o_k(1).
\end{equation}
\\

\noindent {\it Step 2: Lower bound for $\mathcal{F}_\eps(W_k)$.} 
First of all, by the definition of $W_k$ in \eqref{3.4ter} \textcolor{black}{(recall that $W_k$ is supported in $\Rn \setminus B_{r+\overline l}\subset \Rn\setminus B_{\bar r_{\eps}}$, using also \eqref{chooser})}, by Proposition \ref{qineq} and \ref{traceIneq}, using \eqref{heps} and \eqref{bbbb}, we have that 
\begin{equation}\begin{split}\label{upBoundWq}
\bigg|\eps\int_{\mathbb{R}^n}{ h(x)(W_k)_+^{q+1}(x,0)\,dx }\bigg|
\leq\,&\eps\int_{\mathbb{R}^n\setminus B_{\bar{r}_\eps}}{ |h(x)|(W_k)_+^{q+1}(x,0)\,dx }\\
\leq\,&\eps\, \|h\|_{L^m(\mathbb{R}^n\setminus B_{\bar{r}_\eps})}
\|(W_k)_+(\cdot,0)\|_{\Lst}^{q+1}\\
\leq\,& \eps\, C\, \|h\|_{L^m(\mathbb{R}^n\setminus B_{\bar{r}_\eps})}[W_k]_a^{q+1}\\
\leq\,& C\eps^{1+\beta},
\end{split}\end{equation}
where $1+\beta>\alpha/\gamma$. Thus, from \eqref{fprimeW} 
we get that
\begin{equation}\begin{split}\label{Wbound}
&\bigg|\int_{\mathbb{R}^{n+1}_+} {y^a|\nabla W_k|^2\,dX}
-\int_{\mathbb{R}^n}  {(W_k)_+^{\cx} (x,0)\,dx} \bigg|
\\ &\qquad\le \left|\langle \mathcal{F}'_\eps(W_k),W_k\rangle\right| + 
\left|\eps\int_{\mathbb{R}^n}{ h(x) (W_k)_+^{q+1}(x,0)\,dx}\right|\\
&\qquad \leq C\varepsilon^{\alpha/\gamma} + o_k(1).
\end{split}\end{equation}
Moreover, notice that $W_k=U_k$ in $\R^{n+1}_+\setminus B_{r+\overline{l}+1}$ 
(recall \eqref{3.4bis} and \eqref{3.4ter}). 
Hence, using \eqref{contrad} with $\rho:=r+\overline{l}+1$, we get 
\begin{equation}\begin{split}\label{espero}
&\int_{\mathbb{R}^{n+1}_+\setminus B^+_{r+\bar{l}+1}}{y^a|\nabla W_k|^2\,dX}
+\int_{\mathbb{R}^n\setminus\{B_{r+\bar{l}+1}\cap\{y=0\}\}}{(W_k)_+^{2^*_s}(x,0)\,dx}\\
&\qquad =\int_{\mathbb{R}^{n+1}_+\setminus B^+_{r+\bar{l}+1}}{y^a|\nabla U_k|^2\,dX}
+\int_{\mathbb{R}^n\setminus\{B_{r+\bar{l}+1}\cap\{y=0\}\}}
{(U_k)_+^{2^*_s}(x,0)\,dx}
\geq \eta_0,
\end{split}\end{equation}
for $k=k(\rho)$. 
We observe that $k$ tends to $+\infty$ as $\eps\to 0$, 
thanks to \eqref{forse0} and \eqref{eps to zero}. 

From \eqref{espero} we obtain that either 
$$ \int_{\mathbb{R}^n\setminus\{ B_{r+\bar{l}+1} \cap\{y=0\}\} }
{(W_k)_+^{2^*_s}(x,0)\,dx}
\ge\frac{\eta_0}{2}$$
or
$$ \int_{\mathbb{R}^{n+1}_+\setminus B^+_{r+\bar{l}+1}}{y^a|\nabla W_k|^2\,dX}
\ge\frac{\eta_0}{2}.
$$
In the first case, we get that 
$$\int_{\mathbb{R}^n} { (W_k)_+^{2^*_s}(x,0)\,dx}\ge \int_{\mathbb{R}^n\setminus\{B_{r+\bar{l}+1}\cap\{y=0\}\}}
{(W_k)_+^{2^*_s}(x,0)\,dx}\ge 
\frac{\eta_0}{2}.$$
In the second case, taking $\varepsilon$ small (and so $k$ large enough), by \eqref{Wbound} 
we obtain that
\begin{equation*}\begin{split}
\int_{\mathbb{R}^n}{(W_k)_+^{\cx}(x,0)\,dx} & \geq 
\int_{\mathbb{R}^{n+1}_+}{y^a|\nabla W_k|^2\,dX}-C\varepsilon^{\alpha/\gamma}-o_k(1)
\\ &\geq \int_{\mathbb{R}^{n+1}_+\setminus B^+_{r+\bar{l}+1}}
{y^a|\nabla W_k|^2\,dX}-C\varepsilon^{\alpha/\gamma}-o_k(1)
\\ & >\frac{\eta_0}{4}.
\end{split}\end{equation*}
Hence, in both cases we have that 
\begin{equation}\label{lowBoundWp}
\int_{\mathbb{R}^n}{(W_k)_+^{\cx}(x,0)\,dx} >\frac{\eta_0}{4}
\end{equation}
for $\varepsilon$ small and $k$ large enough. We now define $\psi_k:=\alpha_kW_k$, with
$$ \alpha_k^{\cx-2}:=\frac{[W_k]_a^2}{\|(W_k)_+(\cdot,0)\|_{\Lst}^{\cx}}.$$
Notice that from \eqref{fprimeW} we have that
\begin{eqnarray*}
[W_k]_a^2 &\le & \|(W_k)_+(\cdot,0)\|_{\Lst}^{\cx}
+\left|\eps\int_{\R^n}h(x)(W_k)_+^{q+1}(x,0)\,dx\right|
+C\,\varepsilon^{\alpha/\gamma} +o_k(1)\\
&\le & \|(W_k)_+(\cdot,0)\|_{\Lst}^{\cx}
+C\,\varepsilon^{\alpha/\gamma} +o_k(1),
\end{eqnarray*}
where \eqref{upBoundWq} was used in the last line. 
Hence, thanks to \eqref{lowBoundWp}, we get that 
\begin{equation}\label{star-1}
\alpha_k^{\cx-2}\leq 1+C\varepsilon^{\alpha/\gamma}+o_k(1).\end{equation}
Also, we notice that for this value of $\alpha_k$, we have the following chain of identities,
$$[\psi_k]^2_a=\alpha_k^2[W_k]_a^2=\alpha^{\cx}_k
\|(W_k)_+(\cdot,0)\|_{\Lst}^{\cx}
=\|(\psi_k)_+(\cdot,0)\|_{\Lst}^{\cx}.$$
Thus, by Proposition \ref{traceIneq} and~\eqref{isom}, we obtain 
\begin{eqnarray*}
&& S\leq \frac{[\psi_k(\cdot,0)]^2_{\dot{H}^s(\mathbb{R}^n)}}
{\|(\psi_k)_+(\cdot,0)\|_{\Lst}^2}
=\frac{[\psi_k]_a^2}{\|(\psi_k)_+(\cdot,0)\|_{\Lst}^{2}}=\frac{\|(\psi_k)_+(\cdot,0)\|_{\Lst}^{\cx}}
{\|(\psi_k)_+(\cdot,0)\|_{\Lst}^{2}}
=\|(\psi_k)_+(\cdot,0)\|_{\Lst}^{\frac{4s}{n-2s}}.
\end{eqnarray*}
Consequently,
$$\|(W_k)_+(\cdot,0)\|_{\Lst}^{\cx}
=\frac{\|(\psi_k)_+(\cdot,0)\|_{\Lst}^{\cx}}{\alpha_k^{\cx}}
\geq S^{n/2s}\frac{1}{\alpha_k^{\cx}}.$$
This, together with \eqref{star-1}, gives that 
\begin{equation}\begin{split}\label{swkbla}
S^{n/2s}\leq\;&(1+C\varepsilon^{\alpha/\gamma}+o_k(1))^{\frac{\cx}{\cx-2}}
\|(W_k)_+(\cdot,0)\|_{\Lst}^{\cx}\\
\leq\;& \|(W_k)_+(\cdot,0)\|_{\Lst}^{\cx}+C\varepsilon^{\alpha/\gamma}+o_k(1).
\end{split}\end{equation}
We get that 
\begin{eqnarray*}
\mathcal{F}_\eps(W_k)-\frac{1}{2}\langle \mathcal{F}'_\eps(W_k),W_k\rangle
&=&\frac{s}{n}\|(W_k)_+(\cdot,0)\|_{\Lst}^{\cx}\\
&&\qquad +\eps\left(\frac{1}{2}-\frac{1}{q+1}\right)
\int_{\mathbb{R}^n}{h(x)(W_k)_+^{q+1}(x,0)\,dx}\\
&\geq&\frac{s}{n}S^{n/2s} -C\varepsilon^{\beta+1}-C\varepsilon^{\alpha/\gamma}+o_k(1),
\end{eqnarray*}
where we have used \eqref{upBoundWq} to estimate the $(q+1)$-order term.
Finally, using also \eqref{fprimeW} and the fact that $\beta+1>\alpha/\gamma$, we get 
\begin{equation}\label{LowBoundFW}
\mathcal{F}_\eps(W_k)\geq \frac{s}{n}S^{n/2s}-C\varepsilon^{\alpha/\gamma}+o_k(1).
\end{equation}
\\

\noindent {\it Step 3: Lower bound for $\mathcal{F}_\eps(U_k)$.} 
We first observe that by definition
we can write 
\begin{equation}\label{adwetperigyrejh}
U_k=(1-\chi)U_k+\chi U_k=W_k+V_k.\end{equation} 
Therefore
\begin{equation}\begin{split}\label{sumF}
\mathcal{F}_\eps(U_k) = &\, \mathcal{F}_\eps(V_k)+\mathcal{F}_\eps(W_k)
+\int_{\mathbb{R}^{n+1}_+}{y^a\langle\nabla V_k,\nabla W_k\rangle\,dX} \\
&\quad +\frac{1}{\cx}\int_{\mathbb{R}^n}{(V_k)_+^{\cx}(x,0)\,dx}
+\frac{\eps}{q+1}\int_{\mathbb{R}^n}{h(x)(V_k)_+^{q+1}(x,0)\,dx}\\
&\quad +\frac{1}{\cx}\int_{\mathbb{R}^n}{(W_k)_+^{\cx}(x,0)\,dx}
+\frac{\eps}{q+1}\int_{\mathbb{R}^n}{h(x)(W_k)_+^{q+1}(x,0)\,dx}\\
&\quad -\frac{1}{\cx}\int_{\mathbb{R}^n}{(U_k)_+^{\cx}(x,0)\,dx}
-\frac{\eps}{q+1}\int_{\mathbb{R}^n}{h(x)(U_k)_+^{q+1}(x,0)\,dx}.
\end{split}\end{equation}
On the other hand,
\begin{eqnarray*}
&& \int_{\mathbb{R}^{n+1}_+}{y^a\langle\nabla V_k,\nabla W_k\rangle\,dX} \\
&&\qquad = \frac{1}{2}\int_{\mathbb{R}^{n+1}_+}{y^a\langle\nabla U_k-\nabla V_k,\nabla V_k\rangle\,dX} +\frac{1}{2}\int_{\mathbb{R}^{n+1}_+}{y^a\langle\nabla U_k-\nabla W_k,\nabla W_k\rangle\,dX}.
\end{eqnarray*}
Also
\begin{eqnarray*}
&&\langle \mathcal{F}_\varepsilon'(U_k)-\mathcal{F}_\varepsilon'(V_k),V_k
\rangle \\&=& 
\int_{\R^{n+1}_+}y^a\langle\nabla U_k-\nabla V_k, \nabla V_k\rangle\,dX \\
&&\qquad - \eps \int_{\R^n}h(x)(U_k)_+^q(x,0)\,V_k(x,0)\,dx - 
\int_{\R^n}(U_k)_+^{\cx-1}(x,0)\,V_k(x,0)\,dx\\
&&\qquad +\eps \int_{\R^n}h(x)(V_k)_+^{q+1}(x,0)\,dx + 
\int_{\R^n}(V_k)_+^{\cx}(x,0)\,dx,
\end{eqnarray*}
and 
\begin{eqnarray*}
&&\langle \mathcal{F}_\varepsilon'(U_k)-\mathcal{F}_\varepsilon'(W_k),W_k
\rangle \\&=& 
\int_{\R^{n+1}_+}y^a\langle\nabla U_k-\nabla W_k, \nabla W_k\rangle\,dX \\
&&\qquad - \eps \int_{\R^n}h(x)(U_k)_+^q(x,0)\,W_k(x,0)\,dx - 
\int_{\R^n}(U_k)_+^{\cx-1}(x,0)\,W_k(x,0)\,dx\\
&&\qquad +\eps\int_{\R^n}h(x)(W_k)_+^{q+1}(x,0)\,dx + 
\int_{\R^n}(W_k)_+^{\cx}(x,0)\,dx.
\end{eqnarray*}
Hence, plugging the three formulas above into \eqref{sumF} we get
\begin{equation*}\begin{split}
\mathcal{F}_\eps(U_k) = &\, \mathcal{F}_\eps(V_k)+\mathcal{F}_\eps(W_k)+\frac12
\langle \mathcal{F}_\varepsilon'(U_k)-\mathcal{F}_\varepsilon'(V_k),V_k
\rangle +\frac12\langle \mathcal{F}_\varepsilon'(U_k)-\mathcal{F}_\varepsilon'(W_k),W_k
\rangle \\
&\quad +\frac{1}{\cx}\int_{\mathbb{R}^n}{(V_k)_+^{\cx}(x,0)\,dx}
+\frac{\eps}{q+1}\int_{\mathbb{R}^n}{h(x)(V_k)_+^{q+1}(x,0)\,dx}\\
&\quad +\frac{1}{\cx}\int_{\mathbb{R}^n}{(W_k)_+^{\cx}(x,0)\,dx}
+\frac{\eps}{q+1}\int_{\mathbb{R}^n}{h(x)(W_k)_+^{q+1}(x,0)\,dx}\\
&\quad -\frac{1}{\cx}\int_{\mathbb{R}^n}{(U_k)_+^{\cx}(x,0)\,dx}
-\frac{\eps}{q+1}\int_{\mathbb{R}^n}{h(x)(U_k)_+^{q+1}(x,0)\,dx}\\
&\quad +\frac{\eps}{2}
\int_{\R^n}h(x)(U_k)_+^q(x,0)\,V_k(x,0)\,dx + 
\frac12 \int_{\R^n}(U_k)_+^{2^*_s-1}(x,0)\,V_k(x,0)\,dx\\
&\quad -\frac{\eps}{2}\int_{\R^n}h(x)(V_k)_+^{q+1}(x,0)\,dx - 
\frac12\int_{\R^n}(V_k)_+^{\cx}(x,0)\,dx\\
&\quad + \frac{\eps}{2} \int_{\R^n}h(x)(U_k)_+^q(x,0)
\,W_k(x,0)\,dx 
+\frac12 \int_{\R^n}(U_k)_+^{\cx-1}(x,0)\,W_k(x,0)\,dx\\
&\quad -\frac{\eps}{2}\int_{\R^n}h(x)(W_k)_+^{q+1}(x,0)\,dx - \frac12
\int_{\R^n}(W_k)_+^{\cx}(x,0)\,dx.
\end{split}\end{equation*}
Therefore, using~\eqref{boundV} and~\eqref{boundW} we obtain that 
\begin{equation*}\begin{split}
\mathcal{F}_\eps(U_k) \ge &\, \mathcal{F}_\eps(V_k)+\mathcal{F}_\eps(W_k)\\
&\quad +\frac{1}{\cx}\int_{\mathbb{R}^n}{(V_k)_+^{\cx}(x,0)\,dx}
+\frac{1}{\cx}\int_{\mathbb{R}^n}{(W_k)_+^{\cx}(x,0)\,dx}  -\frac{1}{\cx}\int_{\mathbb{R}^n}{(U_k)_+^{\cx}(x,0)\,dx}\\
&\quad +\frac12 \int_{\R^n}(U_k)_+^{\cx-1}(x,0)\,V_k(x,0)\,dx  +\frac12 \int_{\R^n}(U_k)_+^{\cx-1}(x,0)\,W_k(x,0)\,dx  \\
&\quad - \frac12\int_{\R^n}(V_k)_+^{\cx}(x,0)\,dx
 - \frac12 \int_{\R^n}(W_k)_+^{\cx}(x,0)\,dx  \\
 &\quad -\eps\left(\frac{1}{2}-\frac{1}{q+1}\right)\int_{\mathbb{R}^n}{h(x)(W_k)_+^{q+1}(x,0)\,dx}\\
 &\quad-\eps\left(\frac{1}{2}-\frac{1}{q+1}\right)\int_{\mathbb{R}^n}{h(x)(V_k)_+^{q+1}(x,0)\,dx}\\
 &\quad - \frac{\eps}{q+1} \int_{\Rn} h(x) \ukp^{q+1}(x,0)\, dx + \frac{\eps}2 \int_{\Rn} h(x)\ukp^q(x,0)V_k(x,0)\, dx \\  
&\quad + \frac{\eps}2 \int_{\Rn} h(x)\ukp^q(x,0)W_k(x,0)\, dx  -C\varepsilon^{\alpha/\gamma},
\end{split}\end{equation*}
for some positive~$C$. We use identity~\eqref{adwetperigyrejh} to write
\[ \ukp^{\cx-1}(V_k+W_k)=\ukp^{\cx}\quad \mbox{ and }\quad \ukp^{q+1}=\ukp^q(V_k+W_k),\]
and obtain that
\eqlab{\label{qwqweteyryhg}
\mathcal{F}_\eps(U_k) \ge &\, \mathcal{F}_\eps(V_k)+\mathcal{F}_\eps(W_k)\\ 
&\quad+\bigg( \frac12-\frac{1}{\cx}\bigg) \lrq{ \int_{\mathbb{R}^n}(U_k)_+^{\cx}(x,0)\,dx -\int_{\mathbb{R}^n}(V_k)_+^{\cx}(x,0)\,dx 
-\int_{\mathbb{R}^n}(W_k)_+^{\cx}(x,0)\,dx }\\ 
&\quad +  \eps\lr{ \frac{1}{2}-\frac1{q+1} }\int_{\mathbb{R}^n}{h(x)\left[(U_k)_+^{q}(x,0)V_k(x,0) - (V_k)_+^{q+1}(x,0)\right]\,dx} \\ 
&\quad + \eps\lr{ \frac{1}{2}-\frac1{q+1} }\int_{\mathbb{R}^n}{h(x)\left[(U_k)_+^{q}(x,0)W_k(x,0) - (W_k)_+^{q+1}(x,0)\right]\,dx} -C\varepsilon^{\alpha/\gamma}.
} 
Using \eqref{ealpha}, reasoning in the same way for the term with $W_k$ \textcolor{black}{and recalling that $1+\alpha>\alpha/\gamma$ we get that} 
\begin{equation*}\begin{split}
\mathcal{F}_\varepsilon(U_k) \ge  &\, \mathcal{F}_\varepsilon(V_k)+\mathcal{F}_\varepsilon(W_k)\\
&\quad +\frac{s}n \int_{\mathbb{R}^n}\left(
(U_k)_+^{2^*_s}(x,0)-(V_k)_+^{2^*_s}(x,0)-(W_k)_+^{2^*_s}(x,0)\right)\,dx-C\varepsilon^{\alpha/\gamma}\\
= &\, \mathcal{F}_\varepsilon(V_k)+\mathcal{F}_\varepsilon(W_k)\\
&\quad +\frac{s}n \int_{\mathbb{R}^n}(U_k)_+^{2^*_s}(x,0)\left(1-\chi^{2^*_s}(x,0)-(1-\chi(x,0))^{2^*_s}\right)\,dx-C\varepsilon^{\alpha/\gamma},
\end{split}\end{equation*}
where~\eqref{3.4ter} was used in the last line. 
Also, since $\cx>2$ and 
\begin{equation}\label{chipos} 1-\chi^{\cx}(x,0)-(1-\chi(x,0))^{\cx}\ge0 \quad {\mbox{ for any }}x\in\R^n,
\end{equation}
we get
\begin{equation*}\begin{split}
\mathcal{F}_\eps(U_k) \ge  &\mathcal{F}_\eps(V_k)+\mathcal{F}_\eps(W_k) -C\eps^{\alpha/\gamma}. 
\end{split}\end{equation*}
This, \eqref{LowerBoundV} and \eqref{LowBoundFW} imply that 
\begin{equation*}
\mathcal{F}_\eps(U_k)\geq 
\frac{s}{n}S^{n/2s}-c_1\eps^{\alpha/\gamma}- \overline C \eps^{\frac{\cx}{\cx-(q+1)}}+o_k(1).
\end{equation*}
Hence, taking the limit as $k\to+\infty$ we obtain that 
$$ c_\eps=\lim_{k\to+\infty}\mathcal{F}_\eps(U_k)\geq 
\frac{s}{n}S^{n/2s}-c_1\eps^{\alpha/\gamma}  -\overline C \eps^{\frac{\cx}{\cx-(q+1)}},$$
which is a contradiction with assumption (i) of Theorem \ref{PSCthm}. 
This concludes the proof of Lemma \ref{tightness} in the case $n\geq 6s$.
\medskip

Consider now $n\in (2s,6s)$. In such a case, 
one easily sees that
\begin{equation*}\begin{split}
\mathcal{F}_\eps(V_k)&-\frac{1}{2}\langle \mathcal{F}'_\eps(V_k), V_k\rangle 
= \left(\frac{1}{2}-\frac{1}{\cx}\right)\|(V_k)_+(\cdot,0)\|_{\Lst}^{\cx}\\
&\;\;+\eps\left(\frac{1}{2}-\frac{1}{q+1}\right)\int_{\R^n}h(x)(V_k)_+^{q+1}(x,0)\,dx\\
&\geq \eps\left(\frac{1}{2}-\frac{1}{q+1}\right)\int_{\R^n}h(x)(V_k)_+^{q+1}(x,0)\,dx,
\end{split}\end{equation*}
and by \eqref{fprimeV} we get that
\begin{equation}\label{LowerBoundV2}
\mathcal{F}_\varepsilon(V_k)\geq \eps\left(\frac{1}{2}-\frac{1}{q+1}\right)\int_{\R^n}h(x)(V_k)_+^{q+1}(x,0)\,dx-C\eps^{\alpha/\gamma}+o_k(1).
\end{equation}
On the other hand, proceeding analogously to the previous case \textcolor{black}{(check \eqref{swkbla})}, we obtain
\begin{eqnarray*}
\mathcal{F}_\varepsilon(W_k)-\frac{1}{2}\langle \mathcal{F}'_\varepsilon(W_k),W_k\rangle
&=&\frac{s}{n}\|(W_k)_+(\cdot,0)\|_{L^{\cx}(\mathbb{R}^n)}^{\cx}\\
&&\qquad +\varepsilon\left(\frac{1}{2}-\frac{1}{q+1}\right)
\int_{\mathbb{R}^n}{h(x)(W_k)_+^{q+1}(x,0)\,dx}\\
&\geq&\frac{s}{n}S^{n/2s}+\varepsilon\left(\frac{1}{2}-\frac{1}{q+1}\right)
\int_{\mathbb{R}^n}{h(x)(W_k)_+^{q+1}(x,0)\,dx}\\
&&\qquad -C\varepsilon^{\alpha/\gamma}+o_k(1).
\end{eqnarray*}
Thus, using also \eqref{fprimeW}, we get 
\begin{equation}\label{LowBoundFW2}
\mathcal{F}_\varepsilon(W_k)\geq \frac{s}{n}S^{n/2s}+\varepsilon\left(\frac{1}{2}-\frac{1}{q+1}\right)
\int_{\mathbb{R}^n}{h(x)(W_k)_+^{q+1}(x,0)\,dx}-C\varepsilon^{\alpha/\gamma}+o_k(1).
\end{equation}
Now, using the positivity of $h$, from \eqref{qwqweteyryhg} and \eqref{chipos} we get
\bgs{
\mathcal{F}_\eps(U_k) \ge &\, \mathcal{F}_\eps(V_k)+\mathcal{F}_\eps(W_k)\\ 
&\quad+\bigg(\frac12 - \frac{1}{\cx}\bigg) \lrq{\int_{\mathbb{R}^n}(U_k)_+^{\cx}(x,0)\,dx  -\int_{\mathbb{R}^n}(V_k)_+^{\cx}(x,0)\,dx 
-\int_{\mathbb{R}^n}(W_k)_+^{\cx}(x,0)\,dx 
 }\\ 
&\quad -  \eps\lr{ \frac{1}{2}-\frac1{q+1} }\int_{\mathbb{R}^n}{h(x)\left[(V_k)_+^{q+1}(x,0)+(W_k)_+^{q+1}(x,0)\right]\,dx}  -C\varepsilon^{\alpha/\gamma}\\
\ge &\, \mathcal{F}_\eps(V_k)+\mathcal{F}_\eps(W_k)\\ 
&\quad -  \eps\lr{ \frac{1}{2}-\frac1{q+1} }\int_{\mathbb{R}^n}{h(x)\left[(V_k)_+^{q+1}(x,0)+(W_k)_+^{q+1}(x,0)\right]\,dx}  -C\varepsilon^{\alpha/\gamma},\\
\ge &\, \frac{s}{n}S^{n/2s}-C\varepsilon^{\alpha/\gamma}+o_k(1),
} 
where we have used \eqref{LowerBoundV2} and \eqref{LowBoundFW2} in the last line. Passing to the limit as $k\rightarrow\infty$ we reach a contradiction with assumption (i) of Theorem \ref{PSCthm} and thus we finish the proof of Lemma \ref{tightness} in the case $n\in (2s,6s)$.
\end{proof}

\noindent Knowing that the sequence $\{U_k\}_{k\in \N}$ is bounded and tight, one can use the Concentration Compactness principle and prove Theorem \ref{PSCthm}. More precisely, one applies Proposition \ref{CCP} for the positive sequence $\{\ukp\}_{k\in \N}$, which is also bounded and tight, to obtain that
\begin{equation*}\begin{split}
\ukp^{2^*_s}(\cdot,0)&\xrightarrow[k\to +\infty]{} \nu=\overline{U}^{2^*_s}(\cdot,0)+\sum \nu_j\delta_{x_j},\\
y^a|\nabla \ukp|^2&\xrightarrow[k\to +\infty] {}\mu\geq y^a|\nabla\overline{U}|^2+\sum\mu_j\delta_{(x_j,0)},
\end{split}\end{equation*}
and then, following the steps in \cite[Proof of Proposition 4.2.1]{maria} and using Proposition \ref{proppqs}, one deduces $\nu_j=\mu_j=0$ for every $j$. Finally, proceeding as in \cite[Proposition 4.2.1]{maria} (using Proposition \ref{convres} instead of  \cite[Lemma 4.1.1]{maria}) the strong convergence in $\dot{H}^s_a(\mathbb{R}^{n+1}_+)$ follows, and thus Theorem \ref{PSCthm} holds. 

\section{Bound on the minmax value and geometry of the functional}\label{minmax}
\noindent The purpose of this section is to show that the minmax value of the Mountain Pass Lemma lies below the critical threshold given in Theorem \ref{PSCthm}. To see this, the idea is to find a path where the maximum value of the functional is smaller than this critical level (and so the infimum of the maximums along all opportune paths, i.e., the associated minmax value). We obtain such path by working with the fractional Sobolev minimizers, explicitly computed in formula \eqref{SobMin}.\\
One considers, as done in \cite[Section 6.5]{maria}, the ball $B$ given in \eqref{h1}  and takes $\mu_0 >0$ and $\xi \in \Rn$  to be the radius and the center of $B$ respectively. Namely, one has that 
\[\inf_{B_{\mu_0}(\xi)} h >0.\] 
Let $\bar \phi \in C_0^{\infty}(B_{\mu_0}(\xi),[0,1])$ be a cut-off function such that $\bar \phi(x) = 1$ in $B_{\frac{\mu_0}2}(\xi)$. Translating and rescaling the function $z$ in \eqref{SobMin} we define 
\eqlab{\label{zmmu} z_{\mu,\xi}(x):=\mu^{\frac{2s-n}2}z\lr{ \frac{x-\xi}{\mu}},\qquad \mu>0.} 
Let $\bar Z_{\mu,\xi}$ be the extension of $\bar \phi z_{\mu,\xi}$, as defined in  \eqref{extprc}. With some manipulations (check Section \textcolor{black}{6.5} in \cite{maria}), one has that
\eqlab{\label{zsmal1} \|z_{\mu,\xi}\|_{\Lst}^2 =S^{\frac{n-2s}{4s}}} 
and that
\eqlab{\label{zbar1} [\bar Z_{\mu,\xi} ]_a^2 = [\bar \phi z_{\mu,\xi} ]^2_{\dot H^s(\Rn)} \leq S^{\frac{n}{2s}}+C\mu^{n-2s}.}
Moreover, we have the following result.
\begin{lem} There exists $C=C(n,s,\mu_0)>0$ such that
\label{lemzm} \[\|\bar Z_{\mu, \xi} (\cdot, 0)\|_{\Lst}^{\cx} \geq \|z_{\mu,\xi}\|_{\Lst}^{\cx} -C\mu^n.\]
\end{lem}
\begin{proof}
In the next computations, the constant may change value from line to line. Using that  $\bar\phi=1$ on $B_{{\mu_0}/2}(\xi)$, we have that
\bgs{ \|z_{\mu,\xi}\|_{\Lst}^{\cx}-\|\bar Z_{\mu,\xi}(\cdot,0)\|_{\Lst}^{\cx} =&\; \int_{\Rn}(1-\bar \phi^{\cx})z_{\mu,\xi}^{\cx}\, dx   = \int_{\Rn \setminus B_{\frac{\mu_0}2}(\xi) } (1-\bar \phi^{\cx})z_{\mu,\xi}^{\cx}\, dx \\
   \leq &\; \mu^{-n} \int_{\Rn \setminus B_{\frac{\mu_0}2}(\xi)} z^{\cx}\lr{\frac{x-\xi}\mu} \, dx .}
   Making the change of variable $y= (x-\xi)/\mu$ and inserting definition \eqref{SobMin} we get
   \bgs{ \|z_{\mu,\xi}\|_{\Lst}^{\cx}-\|\bar Z_{\mu,\xi}(\cdot,0)\|_{\Lst}^{\cx}   = \int_{\Rn\setminus B_{\frac{\mu_0}{2\mu}}} z^{\cx}(y) \, dy 
    \leq c_{\star}^{\cx}\int_{\Rn\setminus B_{\frac{\mu_0}{2\mu}}} |y|^{-2n}\, dy \leq C \mu^n,} where $C$ depends on $n,s, \mu_0$.
   This proves the lemma.
\end{proof}

\noindent Let $t> 0$. We consider the path $t\bar Z_{\mu,\xi}$ and compute the energy along it. Namely, we focus on  obtaining an upper bound for 
\[ \Fl (t\bar Z_{\mu,\xi}) = \frac{t^2}2 [\bar Z_{\mu,\xi}]_a^2 - \frac{t^{\cx}}{\cx} \|\bar Z_{\mu,\xi}(\cdot,0)\|_{\Lst}^{\cx} - \frac{\eps}{q+1}\int_{\Rn} h(x) \lr{ t\bar Z_{\mu,\xi}(x,0)}^{q+1}\, dx  \] \textcolor{black}{and proving that it stays below the critical threshold  given in Theorem \ref{PSCthm}. Of course, if $t=0$ the energy level is zero, and  for $\eps$ small enough, this is trivially fulfilled.}
We introduce  the following Lemmata.
\begin{lem}\label{lemqp1}
Let $n>\frac{2s(q+1)}{q}$. There exists $\mu^{\star}<\mu_0/2$ such that for any $t> 0$ and any $\mu\in(0,\mu^\star)$  
\bgs{\label{cl3} \int_{\Rn} h(x) \lr{ t \bar Z_{\mu,\xi}(x,0)}^{q+1} \, dx \geq C\; t^{q+1}\;  \mu^{\frac{(2s-n)(q+1)}2+n}  ,}
where $C=C(n,s,h,\mu_0,\mu^\star)$ is a positive constant.\end{lem}
\begin{proof}
Notice that since $q<\cx-1$, we have that $ {\frac{(2s-n)(q+1)}2+n} >0$. 
 \textcolor{black}{Given the definition of $\bar \phi$
 we have that}
\eqlab{ \label{bla12}\int_{\Rn} h(x) \lr{ t \bar Z_{\mu,\xi}(x,0)}^{q+1} \, dx  =&\;  \int_{\Rn} h(x) t^{q+1} \bar \phi^{q+1} z_{\mu,\xi}^{q+1} (x)\, dx \\=&\;\int_{B_{\mu_0}(\xi)}h(x) t^{q+1} \bar \phi^{q+1} z_{\mu,\xi}^{q+1} (x)\, dx  \\ \geq &\; t^{q+1} \inf_{B_{\mu_0}(\xi)} h  \, \int_{B_{\frac{\mu_0}2}(\xi)} z_{\mu,\xi}^{q+1}(x)\, dx,}
 recalling also that $\inf_{B_{\mu_0}(\xi)} h$ is positive. Thus, using \eqref{zmmu}, changing the variable $y=(x-\xi)/\mu$ and inserting definition \eqref{SobMin} we obtain that
\bgs{  \int_{B_{\frac{\mu_0}2}(\xi)} z^{q+1}_{\mu,\xi}(x)\, dx =&\; \mu^{\frac{(2s-n)(q+1)}2+n} \int_{B_{\frac{\mu_0}{2\mu}}} z^{q+1}(y)\, dy \\=&\;   \textcolor{black}{C}\mu^{\frac{(2s-n)(q+1)}2+n}  \int_{B_{\frac{\mu_0}{2\mu}}} (1+|y|^2)^{\frac{(2s-n)(q+1)}2} \, dy,}\textcolor{black}{where $C=C(n,s)>0$.}
Passing to polar coordinates and taking $\mu$ small enough, say $\mu<\mu_0/2$ we get that
\bgs{  \int_{B_{\frac{\mu_0}{2\mu}}}(1+|y|^2)^{\frac{(2s-n)(q+1)}2} \, dy\geq &\; \textcolor{black}{ \int_{B_{\frac{\mu_0}{2\mu}}}|y|^{(2s-n)(q+1)}\, dy} \\
 \geq   &\; c_{n,s} \int_1^{\frac{\mu_0}{2\mu}} \rho^{{(2s-n)(q+1)}}  \rho^{n-1} \, d\rho \\ 
=&\; c_{n,s} \frac{\lr{\frac{\mu_0}{2\mu}}^{(2s-n)(q+1)+n}-1}{(2s-n)(q+1)+n}   .}
We have that $(2s-n)(q+1) +n <0$ and renaming the constants we obtain that
\bgs{  \int_{B_{\frac{\mu_0}{2\mu}}}(1+|y|^2)^{\frac{(2s-n)(q+1)}2} \, dy  \geq c_{n,s} \lr{\mu_0^{(n-2s)(q+1)- n}-(2 \mu)^{(n-2s)(q+1)- n} } \geq  C_{n,s,\mu_0,\mu^\star}, }  for any  $\mu \in  (0,\mu^{\star})$ , $\mu^{\star}<\mu_0/2$, where $C_{n,s,\mu_0,\mu^\star}$ designates a positive constant. Hence 
\bgs{  \int_{B_{\frac{\mu_0}2}(\xi)} z^{q+1}_{\mu,\xi}(x)\, dx \geq  C_{n,s,\mu_0,\mu^\star} \mu^{\frac{(2s-n)(q+1)}2+n},}
and from \eqref{bla12} it follows
\bgs{ \int_{\Rn} h(x) \lr{ t \bar Z_{\mu,\xi}(x,0)}^{q+1} \, dx   \geq &\;C  t^{q+1} \mu^{\frac{(2s-n)(q+1)}2+n} ,}
where $C$ is a positive constant that depends on $n,s, h, \mu_0$ and $\mu^\star$. 
\end{proof}

\noindent Let $\mu^\star$ be fixed as in Lemma \ref{lemqp1}. We want to prove now that the energy level along the path induced by $t\bar Z_{\mu,\xi}$, $t>0$, stays below the critical threshold  given in Theorem \ref{PSCthm} for $\mu<\mu^*$. With this purpose, we state the next result.
\begin{lem}\label{haha}
There exists $\mu_1 \in(0, \mu_0)$ such that
\eqlab{ \label{cl1} \lim_{t\to +\infty} \sup_{\mu \in (0,\mu_1)}\Fl(t\bar Z_{\mu,\xi}) =-\infty.}
Furthermore, if $n>\frac{2s(q+3)}{q+1}$, for any $\mu \in ( 0, \min\{\mu^\star,\mu_1\})$
\eqlab{ \label{cl2} \sup_{t\geq 0} \Fl (t \bar Z_{\mu,\xi}) <\frac{s}n S^{\frac{n}{2s}} +C_1 \mu^{ n-2s} + o(\mu^{n-2s})  - C_3  \eps\mu^{\frac{(2s-n)(q+1)}2+n},}
where $\mu^*$ was given in Lemma \ref{lemqp1}.
\end{lem}
\begin{proof}

Thanks to \eqref{zbar1}, \eqref{zsmal1} and Lemma \ref{lemzm} we have that
\eqlab{ \label{bla11}   \frac{t^2}2 [\bar Z_{\mu,\xi}]^2_a - \frac{t^{\cx}}{\cx} \|\bar Z_{\mu,\xi}(\cdot,0)\|_{\Lst}^{\cx}  \leq &\;\frac{t^2}2 \lr{S^{\frac{n}{2s}} +C_1\mu^{n-2s} } -\frac{t^{\cx}}{\cx} \lr{\|z_{\mu,\xi}\|_{\Lst}^{\cx}  -C_2\mu^n} \\
\leq  &\; \frac{t^2}2 \lr{S^{\frac{n}{2s}} +C_1\mu^{n-2s} }  -  \frac{t^{\cx}}{\cx} \lr{S^{\frac{n}{2s}} -C_2\mu^n}.}
\textcolor{black}{From \eqref{bla12} it follows that for any $\mu\in(0,\mu^*)$}
\[\int_{\Rn} h(x) \lr{ t\bar Z_{\mu,\xi}(x,0)}^{q+1}\, dx  \geq 0,\] and therefore
\bgs{ \Fl (t\bar Z_{\mu,\xi}) = &\;\frac{t^2}2 [\bar Z_{\mu,\xi}]^2_a - \frac{t^{\cx}}{\cx} \|\bar Z_{\mu,\xi}(\cdot,0)\|_{\Lst}^{\cx} - \frac{\eps}{q+1}\int_{\Rn} h(x) \lr{ t\bar Z_{\mu,\xi}(x,0)}^{q+1}\, dx \\
 \leq&\; \frac{t^2}2 \lr{S^{\frac{n}{2s}} +C_1\mu^{n-2s} }  -  \frac{t^{\cx}}{\cx} \lr{S^{\frac{n}{2s}} -C_2\mu^n}.  } 
Now, there exists $\mu_1\in(0,\mu_0)$ small enough such that $S^{n/2s}-C_2\mu^n$ is positive and hence, sending $t $ to $+\infty$ and recalling that $\cx>2$, we obtain
\[ \lim_{t\to +\infty} \Fl (t\bar Z_{\mu,\xi})=-\infty\] for any $\mu \in (0,\mu_1)$. This proves \eqref{cl1}. 

To obtain \eqref{cl2}, we use \eqref{bla11} and taking any  $\mu \in ( 0, \min\{\mu^\star,\mu_1\})$, by Lemma \ref{lemqp1} we have that
\bgs{ \Fl (t\bar Z_{\mu,\xi})   \leq \frac{t^2}2 \lr{S^{\frac{n}{2s}} +C_1\mu^{n-2s} }  -  \frac{t^{\cx}}{\cx} \lr{S^{\frac{n}{2s}} -C_2\mu^n} -C_3 \eps \frac{t^{q+1}}{q+1} \mu^{\frac{(2s-n)(q+1)}2+n} .} By renaming the constants, we obtain
\eqlab{\label{fltx}    \Fl (t\bar Z_{\mu,\xi}) \leq  S^{\frac{n}{2s}} g(t),}
where 
\[g(t):=  \frac{t^2}2 \lr{1 +C_1\mu^{n-2s} }  -  \frac{t^{\cx}}{\cx} \lr{1 -C_2\mu^n} - C_3\eps \frac{t^{q+1}}{q+1}   \mu^{\frac{(2s-n)(q+1)}2+n}   .\] We compute the first derivative of $g$ and have that
\eqlab{ \label{blaaa} g'(t)= t\lrq{ (1+C_1 \mu^{n-2s}) -t^{\cx-2} (1-C_2\mu^n) -C_3 \eps t^{q-1} \mu^{\frac{(2s-n)(q+1)}2+n}} .}
Let 
$$f(t):= (1+C_1 \mu^{n-2s}) -t^{\cx-2} (1-C_2\mu^n)\quad  \hbox{ and } \quad h(t):= C_3 \eps t^{q-1} \mu^{\frac{(2s-n)(q+1)}2+n}.$$ 
Looking for a critical point of $g$ is equivalent to looking for a solution of $f(t)=h(t)$. We notice that $f(t)=0$ has the solution
\[ \alpha =\lr{ \frac{1+C_1 \mu^{n-2s}} { 1-C_2\mu^n}}^{\frac{n-2s}{4s}},\] which is positive for $\mu\in(0,\mu_1)$. Moreover, $f$ is strictly decreasing on $(0,+\infty)$ for any $\mu \in (0,\mu_1) $, $h$ is strictly increasing  on $(0,+\infty)$ (recalling that $q\textcolor{black}{\geq}1$) and 
\bgs{ f(0)>0, \quad  f(\alpha)=0 , \quad \mbox{and} \quad h(0)=0 , \quad h(\alpha)>0.}
From this it follows that there exists (and is unique) $t_{\mu} \in (0,\alpha)$ such that $f(t_{\mu})=h(t_{\mu})$ (hence $g'(t_{\mu})=0$). Notice also that $g'(t)>0$ on $(0,t_{\mu})$ and   $g'(t)<0$ on $(t_{\mu},+\infty)$. This implies that $g(t_{\mu})$ is a maximum. 
Now, denoting by 
\[F(t):= \frac{t^2}2 \lr{1 +C_1\mu^{n-2s} }  -  \frac{t^{\cx}}{\cx} \lr{1 -C_2\mu^n} \]
we have that $F'(t)= tf(t) >0$ on $(0,\alpha)$, hence $F(t_{\mu})\leq F(\alpha)$. 
  On the other hand, $t_{\mu}>0$ and there exists $\delta >0$ independent on $\eps$ and $ \mu$ such that $t_{\mu}\geq \delta$.
  Indeed, since $g'(t_{\mu})=0$, one has from \eqref{blaaa} that
  \[ 1<1 + C_1 \mu^{n-2s}  = t_{\mu}^{\cx-2}  (1-C_2\mu^n)+ C_3 \eps t_{\mu}^{q-1} \mu^{\frac{(2s-n)(q+1)}2+n} <t_{\mu}^{\cx-2}  + C_3 t_{\mu}^{q-1} \] for any $\mu \in (0,\mu_1) $ and  $\eps \in (0,1)$ and this implies the claim. 
  And so by renaming $C_3$ (that will depend on $\delta$ also) \textcolor{black}{and computing $F(\alpha)$} we have that 
  \bgs{ g(t)\leq  g(t_{\mu} )=&\;F(t_{\mu} )  - C_3 \eps \frac{t_{\mu} ^{q+1}}{q+1} \mu^{\frac{(2s-n)(q+1)}2+n}  \leq F(\alpha)- C_3  \eps\mu^{\frac{(2s-n)(q+1)}2+n}   \\
\leq &\; \lr{\frac{1}2-\frac{1}{\cx} } (1+C_1\mu^{n-2s})^{\frac{n}{2s}} (1-C_2\mu^{n})^{\frac{2s-n}{2s}}  - C_3  \eps\mu^{\frac{(2s-n)(q+1)}2+n}    \\
=&\; \frac{s}n +C_1 \mu^{ n-2s} + o(\mu^{n-2s})  - C_3  \eps\mu^{\frac{(2s-n)(q+1)}2+n} .} 
Renaming the constants, from \eqref{fltx} we have that \textcolor{black}{for any $\mu \in (0,\min\{\mu^*,\mu_1\})$}
 \[ \Fl (t\bar Z_{\mu,\xi}) \leq \frac{s}n S^{\frac{n}{2s}} +C_1 \mu^{ n-2s} + o(\mu^{n-2s})  - C_3  \eps\mu^{\frac{(2s-n)(q+1)}2+n}.\] \textcolor{black}{This concludes the proof of Lemma \ref{haha}.}
\end{proof}

\section{Proof of Theorem \ref{thm}}\label{proofThm}

\noindent Let us take $\mu=\eps^\beta$ with $\beta$ satisfying
\eqlab{ \label{beta1}\frac{2}{n(q+1)-2s(q+3)} <\beta<\frac{\delta}{\frac{(2s-n)(q+1)}{2}+n},}
and $\delta>0$ large enough to have both conditions satisfied
\textcolor{black}{(notice that both denominators are positive by hypothesis)}. This gives in particular that
\[\beta(n-2s)>1+\beta\left[ \frac{(2s-n)(q+1)}2+n \right].\] 
\noindent\textcolor{black}{ Consider now the case $n\in (2s,6s)$.} For $\varepsilon$ small enough, from Lemma \ref{haha} and renaming the constants, we obtain
\begin{equation*}\begin{split}
c_\varepsilon+C\varepsilon^{1+\delta}&\leq \frac{s}{n}S^{n/2s}+C\varepsilon^{\beta(n-2s)}+o(\varepsilon^{\beta(n-2s)})-C\varepsilon^{1+\beta\left({\frac{(2s-n)(q+1)}{2}+n}\right)}\\
&\leq \frac{s}{n}S^{n/2s}-C\varepsilon^{1+\beta\left({\frac{(2s-n)(q+1)}{2}+n}\right)}<\frac{s}{n}S^{n/2s},
\end{split}\end{equation*}
that is assumption (i) of Theorem \ref{PSCthm} for $n\in(2s,6s)$. 
\medskip

On the other hand, if $n\geq 6s$ we have that
$$q\geq 1 >\frac{n}{n-s}\frac{n}{n-2s}-1,$$
\textcolor{black}{which assures that \[ \frac{2}{n(q+1)-2s(q+3)} <\frac{\cx}{(n-2s)[\cx-(q+1)]} .\] So we pick now $\beta$ with the additional condition \bgs{ \frac{2}{n(q+1)-2s(q+3)}<\beta<\frac{\cx}{(n-2s)[\cx-(q+1)]} }(still taking $\delta>0$ such that \eqref{beta1} is satisfied).
In particular we have that \[ 1+\beta\left[\frac{(2s-n)(q+1)}2+n \right]< \frac{\cx}{\cx-(q+1)}.\] 
Therefore for $\eps$ small enough, we get from Lemma \ref{haha} that
\bgs{ c_\eps+c_1\eps^{1+\delta}+\overline c\eps^{\frac{\cx}{\cx-(q+1)}}<&\; \frac{s}n +C\eps^{\beta(n-2s)} -C \eps^{1+\beta\left({\frac{(2s-n)(q+1)}{2}+n}\right)}\\ <&\;\frac{s}{n}S^{n/2s} -C  \varepsilon^{1+\beta\left({\frac{(2s-n)(q+1)}{2}+n}\right)} <\frac{s}{n}S^{n/2s},}
that is assumption (i) of Theorem \ref{PSCthm} for $n>6s$.}
%

\noindent Hence, Theorem \ref{PSCthm} yields that the operator $\Fl$ satisfies the Palais-Smale condition. Moreover, Lemma \ref{haha} assures that it has geometry of Mountain Pass and therefore we conclude the existence of a critical point of $\Fl$. According to the considerations made at the end of Section \ref{intr}, this implies the existence of a positive solution of \eqref{problem} and concludes the proof of Theorem \ref{thm}.


\begin{thebibliography}{1}
\frenchspacing

\bibitem{ABC} {\sc A. Ambrosetti, H. Brezis, G. Cerami}: Combined effects of concave and 
convex nonlinearities in some elliptic problems. {\it J. Funct. Anal.} {\bf 122} (2) (1994), 519--543.

\bibitem{AGP1}  {\sc A. Ambrosetti, J. Garcia Azorero, I. Peral}: 
Perturbation of $\Delta u+u^{(N+2)/(N-2)}=0$, the Scalar Curvature Problem in~$\R^N$, and related Topics. 
{\it J. Funct. Anal.} {\bf 165} (1998), 112--149. 

\bibitem{AGP} {\sc A. Ambrosetti, J. Garcia Azorero, I. Peral}: 
Elliptic variational problems in $\R^n$ with critical growth. 
{\it J. Differential Equations} {\bf 168} (2000), 10--32. 

\bibitem{ALM} {\sc  A. Ambrosetti, Y.Y. Li, A. Malchiodi}: On the Yamabe problem and the scalar 
curvature problems under boundary conditions. {\it Math. Ann.} {\bf 322} (2002), no. 4, 667--699.

\bibitem{mps} {\sc A. Ambrosetti, P. Rabinowitz}: Dual variational methods 
in critical point theory and applications. {\it J. Funct. Anal.} {\bf 14} 
(1973), 349--381.

\bibitem{bego}{\sc B. Barrios, E. Colorado, R. Servadei, F. Soria}: A critical fractional equation 
with concave-convex nonlinearities. 
{\it Ann. Inst. H. Poincar{\'e} Anal. Non Lin{\'e}aire}. {\bf 32} 
(2015), 875--900.



\bibitem{brezisniren} {\sc H. Brezis, L. Nirenberg}:
Positive solutions of nonlinear elliptic equations involving critical
Sobolev exponents. {\it Comm. Pure Appl. Math.} {\bf 36} (1983),
no. 4, 437--477.



\bibitem{brezis}{\sc H. Brezis}: Functional analysis, {S}obolev spaces and partial differential
              equations. Springer, New York (2011).
              

\bibitem{nonlocal} {\sc C. Bucur, E. Valdinoci}: Nonlocal diffusion and applications. {\it Lecture Notes of the Unione Matematica Italiana}, Springer {\bf 20} (2016).


\bibitem{extension} {\sc L. Caffarelli, L. Silvestre}:
An extension problem related to the fractional Laplacian.
{\it Comm. Partial Differential Equations} {\bf 32} (2007), 1245--1260.

\bibitem{CW} {\sc F. Catrina, Z.-Q. Wang}: Symmetric solutions for the prescribed scalar curvature 
problem. {\it Indiana Univ. Math. J}. {\bf 49} (2000), no. 2, 779--813.

\bibitem{Cing} {\sc S. Cingolani}: Positive solutions to perturbed elliptic problems in $\R^N$ 
involving critical Sobolev exponent. {\it Nonlinear Anal.} {\bf 48} (2002), no. 8, 
Ser. A: Theory Methods, 1165--1178.

\bibitem{costi} {\sc A. Cotsiolis, N. Tavoularis}:
Best constants for Sobolev inequalities for higher order fractional
derivatives. {\it J. Math. Anal. Appl.} {\bf 295} (2004),
225--236.

\bibitem{hitch} {\sc E. Di Nezza, G. Palatucci, E. Valdinoci}:
Hitchhiker's guide to the fractional Sobolev spaces.
{\it Bull. Sci. Math.} {\bf 136} (2012), no. 5, 521--573.

\bibitem{DMPV} {\sc S. Dipierro, M. Medina, I. Peral, E. Valdinoci}: 
Bifurcation results for a fractional elliptic
equation with critical exponent in~$\R^n$. {\it Manuscripta Mathematica} {\bf 153} (2017), no. 1-2, 183--230. 

\bibitem{maria}{\sc S. Dipierro, M. Medina, E. Valdinoci}: Fractional elliptic problems with critical growth in the whole of $\mathbb{R}^n$. { \it  Lecture Notes Scuola Normale Superiore di Pisa}, Springer {\bf 15} (2017).
 
\bibitem{lions1} {\sc P. L. Lions}: The concentration-compactness principle in the calculus of variations. 
The limit case. I. 
{\it Rev. Mat. Iberoamericana 1} (1985), no. 1, 145--201.

\bibitem{lions2} {\sc P. L. Lions}: The concentration-compactness principle in the calculus of variations. 
The limit case. II. 
{\it Rev. Mat. Iberoamericana 1} (1985), no. 2, 45--121.

\bibitem{SV-2} {\sc R. Servadei, E. Valdinoci}: Mountain pass solutions for non-local elliptic operators. 
{\it J. Math. Anal. Appl.} {\bf 389} (2012), no. 2, 887--898.

\bibitem{SV-1} {\sc R. Servadei, E. Valdinoci}: A Brezis-Nirenberg result for non-local 
critical equations in low dimension. {\it Commun. Pure Appl. Anal.} {\bf 12} (2013), no. 6, 2445--2464.

\bibitem{silvestre}{\sc L. Silvestre}: Regularity of the obstacle problem for a fractional power of the {L}aplace operator. {\it Comm. Pure Appl. Math.} {\bf 60} (2007), no. 1, 67--112.


\end{thebibliography}
\end{document}